\documentclass[journal]{IEEEtran}
\usepackage[nooneline,flushleft]{caption2}
\usepackage{mathrsfs}
\usepackage{graphicx,subfigure}
\usepackage{color}
\usepackage{amsfonts}
\usepackage{amssymb}
\usepackage{cite}
\usepackage[all]{xy}
\usepackage{bm}
\usepackage{multirow}
\usepackage{algorithm}
\usepackage{algorithmic}
\newtheorem{corollary}{Corollary}

\newtheorem{assumption}{Assumption}
\newtheorem{remark}{Remark}
\newtheorem{lemma}{Lemma}
\newtheorem{theorem}{Theorem}
\newcommand{\BEASN}{\begin{eqnarray*}}
\newcommand{\EEASN}{\end{eqnarray*}}
\newcommand{\BEAS}{\begin{eqnarray}}
\newcommand{\EEAS}{\end{eqnarray}}
\newcommand{\BEQ}{\begin{equation}}
\newcommand{\EEQ}{\end{equation}}
\newcommand{\BIT}{\begin{itemize}}
\newcommand{\EIT}{\end{itemize}}
\newcommand{\eg}{{e.g.}}
\newcommand{\ie}{{i.e.}}
\newcommand{\nn}{{\nonumber}}

\newcommand{\ones}{\bm{1}}
\newcommand{\reals}{\mathbb{R}}
\newcommand{\expect}{\mathbb{E}}

\newcommand{\bmx}{\bm{x}}

\newcommand{\bmu}{\bm{u}}

\newcommand{\bmy}{\bm{y}}
\newcommand{\bmz}{\bm{z}}
\newcommand{\bme}{\bm{e}}
\newcommand{\grad}{\nabla}
\newcommand{\grade}{\breve{\bm{g}}}
\newcommand{\reg}{\bm{Reg}}
\newcommand{\wm}{\bm{P}}
\newcommand{\bigo}{\mathcal{O}}
\newcommand{\breg}{V}
\newcommand{\calK}{\mathcal{K}}

\newcommand{\calG}{\mathcal{G}}
\newcommand{\calV}{\mathcal{V}}
\newcommand{\calE}{\mathcal{E}}

\newcommand{\pr}{\textcolor[rgb]{0.00,0.00,0.00}}
\newcommand{\dy}{\textcolor[rgb]{0.00,0.00,0.00}}

\begin{document}
\title{Distributed Mirror Descent for Online Composite Optimization}
\author{Deming~Yuan,~\IEEEmembership{Member,~IEEE},~Yiguang~Hong,~\IEEEmembership{Fellow,~IEEE},
~Daniel~W.~C.~Ho,~\IEEEmembership{Fellow,~IEEE},~and~Shengyuan~Xu
\thanks{D. Yuan and S. Xu are with the School of Automation, Nanjing University of Science and Technology, Nanjing 210094, China (e-mail: dmyuan1012@gmail.com; syxu@njust.edu.cn).}
\thanks{Y. Hong is with the Key Laboratory of Systems and Control, Academy of Mathematics and Systems Science, Chinese Academy of Sciences, Beijing 100190, China (e-mail: yghong@iss.ac.cn).}
\thanks{D. W. C. Ho is with the Department of Mathematics, City University of Hong Kong, Kowloon, Hong Kong (e-mail: madaniel@cityu.edu.hk).}
}

\maketitle

\begin{abstract}
In this paper, we consider an online distributed composite optimization problem over a time-varying multi-agent network that consists of multiple interacting nodes, where the objective function of each node consists of two parts: a loss function that changes over time and a regularization function. This problem naturally arises in many real-world applications ranging from wireless sensor networks to signal processing. We propose a class of online distributed optimization algorithms that are based on approximate mirror descent, which utilize the Bregman divergence as distance-measuring function that includes the Euclidean distances as a special case. We consider two standard information feedback models when designing the algorithms, that is, full-information feedback and bandit feedback. For the full-information feedback model, the first algorithm attains \dy{an average regularized regret of order $\mathcal{O}(1/\sqrt{T})$ with the total number of rounds $T$}. The second algorithm, which only requires the information of the values of the loss function at two predicted points instead of the gradient information, achieves the same average regularized regret as that of the first algorithm. Simulation results of a distributed online regularized linear regression problem are provided to illustrate the performance of the proposed algorithms.
\end{abstract}

\section{Introduction}
In recent years, there have been considerable research efforts on distributed multi-agent optimization,
due to its widespread applications in machine learning, sensor networks, smart grids, and distributed control systems. In distributed multi-agent optimization, each node is endowed with a local private objective function, and the main task of the network is to collectively minimize the sum of the objective functions of nodes by local computations and communications
\cite{nedic2010,ram2010,zhu2011tac,zhao2017tac,lin2017tac,lu2011tac,hong2014tac,yang2017tac}.

There exist many efficient distributed optimization algorithms in the literature. In the seminal work \cite{nedic2010}, the authors proposed the notable distributed projected subgradient algorithm for solving distributed constrained multi-agent optimization problem and provided its convergence analysis results. To establish non-asymptotic convergence results, Duchi et al.\cite{duchi2011} proposed a distributed optimization algorithm that is based on dual averaging, and characterized its explicit convergence rate. The works\cite{xi2014arxiv,rabbat2015,deming2018auto} developed a class of distributed optimization algorithms that are built on mirror descent, which generalize the projection step by using the Bregman divergence. Different from the aforementioned works that deal only with non-composite objective functions, the authors in \cite{shi2015tsp,zeng2017scl} considered a decentralized composite optimization problem where the local objective function of every node is composed of a smooth function and a nonsmooth regularizer. This problem naturally arises in many real applications including distributed estimation in sensor networks \cite{hosseini2013cdc,yuan2016siam}, distributed quadratic programming \cite{shi2015tsp}, and distributed machine learning \cite{jakovetic2014tac,mateos2010}, to name a few. It is worth emphasizing that the objective functions considered in the aforementioned works are \emph{time-invariant}. However, in many real applications the objective functions change over time, due to the dynamically changing and uncertain nature of the environment, taking the distributed estimation in sensor networks as an example \cite{hosseini2013cdc}. Online optimization is known as a powerful tool that can deal with time-varying cost functions that satisfy certain properties (see, \eg, \cite{duchi2010colt,flaxman2005,agarwal2010,shamir2017,zinkevich2003}).
The work \cite{zinkevich2003} considered online convex optimization where the objective function varies over time, and introduced a notion of regret to measure the performance of online optimization algorithms. Based on \cite{zinkevich2003}, a class of bandit online optimization algorithms were proposed in \cite{flaxman2005,agarwal2010,shamir2017} to remove the need for gradient information of the objective functions.

Building on the distributed optimization model in \cite{nedic2010}, in this paper we focus on solving distributed composite optimization problem in \emph{online} setting over a time-varying network. Specifically, the problem is as the following:
\dy{
\BEQ
\begin{array}{lll}
\mathop{\min}_{\bmx\in\calK}     &  & \sum\limits_{t=1}^{T}\sum\limits_{i=1}^{m} f_{i,t} (\bmx)  \\
\mathrm{where}  &  & f_{i,t} (\bmx) := \ell_{i,t} (\bmx) + r(\bmx)
\end{array}
\label{Online-CO}
\EEQ
where $\ell_{i,t} : \reals^d \rightarrow \reals$ is the convex loss (or cost) function associated with node $i$ at time $t$, $r$ is the convex and possibly nonsmooth regularization function, and $\calK \subseteq \reals^d$ is the convex constraint set (or decision space) known to all the nodes in the network.
}



Recently, there have been increasing research interests in solving distributed convex optimization in online setting \cite{hosseini2013cdc,hosseini2016tac,jadbabaie2018tac,lee2017tac,yan2013tkde,raginsky2011acc}. In the work\cite{yan2013tkde}, the authors developed a distributed autonomous online learning algorithm that is based on computing local subgradients, and they derived an \dy{$\mathcal{O}(\ln T/T)$ average regret} rate for strongly convex cost functions. On the other hand, the work \cite{hosseini2013cdc} extended the distributed dual averaging algorithm in \cite{duchi2011} to online setting, and derived an \dy{$\mathcal{O}(1/\sqrt{T})$ average regret} rate. The authors in \cite{hosseini2016tac} further applied the online distributed dual averaging algorithm to dynamic networks. The authors in \cite{jadbabaie2018tac} proposed an online distributed optimization that is based on mirror descent and established its convergence analysis results.



Our goal in the paper is to design distributed algorithms for solving problem (\ref{Online-CO}) under full-information feedback and bandit feedback that attain non-trivial performance, respectively. For each feedback model, we devise a distributed algorithm with sub-linear cumulative regularized regret. Our algorithms are simple and naturally involve in each round a local approximate mirror descent step and a communication and averaging step where nodes aim at aligning their decisions to those of their neighbors. The average regularized regret analysis of our algorithms is however challenging as it requires us to understand how these two steps interact.
In particular, this paper aims at establishing average regularized regrets that are comparable to those of the centralized algorithms. Specifically, the contributions of the paper can be summarized in three directions.

\BIT
\item[$\bullet$]
First, we have developed two mirror descent based distributed online algorithms that generalize the standard gradient descent based algorithms. In particular, we analyze approximate versions of distributed online algorithms where the decisions at every round are not the exact minimizer of the corresponding optimization problem but approximate ones; this is commonly observed in iterative optimization problems, since in general they cannot be solved to infinite precision. We have also highlighted the dependence of the average regret bound on the optimization errors. Different from the work \cite{hosseini2016tac}, our algorithms are non-Euclidean, in the sense that they enable us to generate more efficient updates by carefully choosing the Bregman divergence. Different from the work \cite{jadbabaie2018tac}, our optimization problem is composite and the proposed algorithms are approximate.
\item[$\bullet$]
Second, in the case of full-information feedback, we develop an online distributed composite mirror descent ($\mathsf{ODCMD}$) algorithm and show that \dy{the average regularized regret of algorithm $\mathsf{ODCMD}$ scales as $\mathcal{O}(1/\sqrt{T})$}. Algorithm $\mathsf{ODCMD}$ extends the algorithm in \cite{duchi2010colt} to distributed multi-agent setting, and in particular, the \dy{$\mathcal{O}(1/\sqrt{T})$ average regret} scaling is identical to those of centralized online optimization algorithms proposed in \cite{zinkevich2003,duchi2010colt}.
Moreover, this rate of convergence is the same as that of \cite{hosseini2016tac}, where the objective function is non-composite and the algorithm is gradient descent based. Different from the works \cite{shi2015tsp,chen2012allerton,zeng2017scl}, where the objective functions are time-invariant and the algorithm is gradient descent based, algorithm $\mathsf{ODCMD}$ is online and based on mirror descent.
\item[$\bullet$]
Third, we further solve the online distributed composite optimization under bandit feedback. New challenges arise in the absence of gradient information when designing the algorithm, however, we remedy this by introducing a distributed gradient estimator that only needs two point observations of the loss function.
\pr{
Specifically, we propose a bandit online distributed composite mirror descent ($\mathsf{BanODCMD}$) algorithm and show that the average regularized regret scales as $\mathcal{O}( p(d) d / \sqrt{T})$, where $p(d)$ is a constant that depends on the norm used. In the case of Euclidean norm, we have $p(d) = 1$ and the average regularized regret scales as $\mathcal{O}(d)$ with the dimension $d$, which is identical to that of centralized online bandit optimization \cite{flaxman2005}. In the case of $\ell_p$ norm with $p\geq 1$, we have $p(d) \leq \sqrt{d}$ and algorithm $\mathsf{BanODCMD}$ achieves an average regret scaling as $\mathcal{O}( d \sqrt{d} / \sqrt{T})$, which is a factor of $\sqrt{d}$ better than that obtained for a centralized algorithm in \cite{agarwal2010}.
}
Different from existing algorithms \cite{hosseini2016tac,jadbabaie2018tac,lee2017tac,yan2013tkde}, the proposed algorithm $\mathsf{BanODCMD}$ removes the need for gradient information of the loss functions. This makes our algorithm applicable to cases where the gradient information is unavailable or costly to access.
\EIT

The remainder of this paper is organized as follows: In Section II, we give a description of the notations and assumptions used throughout the paper. In Section III, we propose the $\mathsf{ODCMD}$ algorithm that solves the problem under full-information feedback and derive its average regularized regret rate. In Section IV, to solve the problem under bandit feedback, we propose the $\mathsf{BanODCMD}$ algorithm and establish its average regularized regret rate. We numerically evaluate the convergence performance of $\mathsf{ODCMD}$ and $\mathsf{BanODCMD}$ on a distributed online regularized linear regression problem in Section V. Finally, we conclude with Section VI.

\emph{Notation and Terminology:} Let $\reals^{d}$ be the $d$-dimensional vector space. Write $\left< \bm{a},\bm{b} \right>$ to denote the standard inner product on $\reals^{d}$, for any $\bm{a},\bm{b}\in\reals^d$. Write $\| \bmx \|_2$ to denote the Euclidean norm of a vector $\bmx\in\reals^d$. Denote $\| \cdot \|_\ast$ the dual norm to $\| \cdot \|$, defined by $\| \bm{x} \|_\ast = \max_{\| \bm{y}  \| = 1} \left< \bm{x},\bm{y}\right>$.
For a vector $\bmx\in\reals^d$, let $[\bmx]_i$ be the $i$th entry of $\bmx$. Write $[m]$ to denote the set of integers $\{1,\ldots,m\}$. Let $\Delta_d$ be the probability simplex in $\reals^{d}$, \ie, $\Delta_d = \{ \bmx \in \reals^{d} \mid \sum_{i=1}^{d} [\bmx]_i = 1, [\bmx]_i\geq 0,i\in[d]\}$. For a matrix $\bm{P}$, use $[\bm{P}]_{ij}$ to denote the entry of $i$th row and $j$th column. Use the notation $\nabla f(\bmx)$ to refer to any (sub-)gradient of $f$ at point $\bmx$. Given two positive sequences $\{a_t\}_{t=1}^{\infty}$ and $\{b_t\}_{t=1}^{\infty}$, write $a_t = \mathcal{O} (b_t)$ if $\lim\sup_{t\rightarrow\infty} \frac{a_t}{b_t}  < \infty$. Write $\expect[X]$ to denote the expected value of a random variable $X$.

\section{Problem Setting}
\subsection{The Problem}
In this paper we focus on solving problem (\ref{Online-CO}) over a time-varying network $\calG_t = (\calV,\calE_t)$, where $\calV = \{ 1,\ldots,m\}$ is the node set and $\calE_t = \{  \{i,j\} \mid [\wm(t)]_{ij} > 0, i,j\in\calV \}$ is the set of activated edges at time $t$ with $\wm(t) = [\wm(t)]_{ij}\in\reals^{m\times m}$ defined to be the weight matrix that represents the communication pattern of the network at time $t$.

The function $r$ in problem (\ref{Online-CO}) serves as a fixed \emph{regularization function} of each node and is typically used to promote certain structure types in the solution of the problem or control the complexity of the solution. Examples of the regularization term $r$ include: i) $\ell_1$-regularization: $r(\bmx) = \lambda \| \bmx\|_1$ for some $\lambda>0$, which can be used to promote the sparsity of the solution in distributed estimation in sensor networks \cite{yuan2016siam,hosseini2013cdc}; and ii) mixed regularization: $r(\bmx) = \frac{\sigma}{2} \| \bmx\|_2^2 + \lambda \| \bmx\|_1$ for some $\sigma > 0$ and $\lambda>0$, which is used in distributed elastic net regression problem \cite{rabbat2015}.

In general, the network of nodes interacts with the environment according to the following protocol. To be specific, at round $t=1,2,3,\ldots$,
\BIT
\item[$\bullet$] Node $i\in\calV$ makes a decision $\bmx_{i,t}\in\calK$;
\item[$\bullet$] The environment selects the loss function $\ell_{i,t}$, and node $i$ receives a signal about the loss function $\ell_{i,t}$;
\item[$\bullet$] Node $i$ communicates the information with its instant neighbors.
\EIT

The objective of every node $i\in\calV$ in the network is to generate a sequence of estimates $\{ \bmx_{i,t} \}_{t=1}^{T} \in \calK$ that minimizes the \emph{average regularized regret} over $T$ rounds, defined by:
\dy{
\BEAS
\overline{\reg}_{i}(T) &:=&  \frac{1}{T} \sum_{t=1}^{T} \sum_{j=1}^{m}  \left( \ell_{j,t} (\bmx_{i,t}) + r(\bmx_{i,t})  \right) \nn\\
&&- \frac{1}{T} \sum_{t=1}^{T} \sum_{j=1}^{m} \left( \ell_{j,t} (\bmx^\star) + r(\bmx^\star) \right)
\label{regularized-regret}
\EEAS
where
$\bmx^\star  =  \mathop{\arg\min}_{\bmx\in\calK}  \sum_{t=1}^{T} \sum_{j=1}^{m} \left( \ell_{j,t} (\bmx) + r(\bmx) \right).$ The average regularized regret measures the difference between the average loss of every node $i$'s decisions $\{ \bmx_{i,t} \}_{t=1}^{T}$ and the average loss of the best constant decision $\bmx^\star$ chosen in hindsight.
}

\subsection{Feedback Models}
This paper considers the following two different information feedback models in solving problem (\ref{Online-CO}):
\BIT
\item[$\bullet$]
\emph{Full-information feedback:} After each node $i$ has committed to the decision $\bmx_{i,t}$, a loss function $\ell_{i,t}$ along with its entire information is revealed to node $i$; in particular, each node $i$ can use the gradient information of its loss function $\ell_{i,t}$ to construct the next decision.
\item[$\bullet$]
\emph{Bandit feedback:} Only the value of the loss function $\ell_{i,t}$ at (or near) the committed decision is revealed to each node $i$, no other information about $\ell_{i,t}$ is revealed to node $i$.
\EIT

\subsection{Assumptions}
Throughout, we make the following standard assumptions on the network $\calG_t$ and the associated weight matrix $\wm(t)$.
\begin{assumption}\label{assump:network}
The network $\calG_t = (\calV,\calE_t)$ and the weight matrix $\wm(t)$ satisfy the following.
\BIT
\item[(a)]
$\wm(t)$ is doubly stochastic for all $t\geq 1$, that is, $\sum_{j=1}^{m} [\wm(t)]_{ij} = 1$ and $\sum_{i=1}^{m} [\wm(t)]_{ij} = 1$, for all $i,j\in\calV$.
\item[(b)]
There exists a scalar $\zeta > 0$ such that $[\wm(t)]_{ii} \geq \zeta$ for all $i$ and $t\geq 1$, and $[\wm(t)]_{ij} \geq \zeta$ if $\{i,j\}\in\calE_t$.
\item[(c)]
There exists an integer $B \geq 1$ such that the graph $(\calV,\calE_{kB+1} \cup \cdots \cup \calE_{(k+1)B})$ is strongly connected for all $k\geq 0$.
\EIT
\end{assumption}

\pr{
The network model in Assumption \ref{assump:network} is widely used in distributed multi-agent optimization community (see, e.g., \cite{nedic2010,yuan2016siam,lee2017tac}). It is easy to achieve in a distributed setting. For example, when bidirectional communication between nodes is allowed, Assumption 1(a) (i.e., doubly stochasticity) follows by enforcing symmetry on the node interaction matrix. Assumption 1(b) requires that each node assigns a significant weight to its own decision and those of its neighbors. Assumption 1(c) simply states that the network is frequently connected, but need not be connected at every time instant. Assumption 1 includes the fixed and connected network as a special case, by taking $B=1$.
}

We make the following assumptions on the constraint set and functions in problem (\ref{Online-CO}).
\begin{assumption}\label{assump-set-bounded}
The decision space $\calK$ has a finite diameter, that is, $\forall \bmx,\bmy\in\calK$, $ \|  \bmx - \bmy \| \leq D_{\calK}$.
\end{assumption}

\begin{assumption}\label{assump-fcn-obj}
\pr{
Function $\ell_{i,t}$ ($i\in[m]$ and $t\in[T]$) is $G_\ell$-Lipschitz over $\calK$, that is,
\BEASN
| \ell_{i,t}(\bmx)  -  \ell_{i,t}(\bmy) | \leq G_\ell \| \bmx - \bmy\|, \quad \forall \bmx,\bmy\in\calK.
\EEASN
This, in fact, implies that the (sub-)gradient of $\ell_{i,t}$ is uniformly bounded by the same constant $G_\ell$, that is, $\| \grad \ell_{i,t} (\bmx) \|_\ast \leq G_\ell$, for all $\bmx\in\calK$. In addition, function $r(\bmx)$ is $G_r$-Lipschitz over $\calK$.
}
\end{assumption}

\begin{remark}
It is worth noting that Assumption \ref{assump-set-bounded} is standard in solving distributed or even centralized online convex optimization problems (see, \eg, \cite{agarwal2010,yan2013tkde,lee2017tac}).
In Assumption \ref{assump-fcn-obj}, $\ell_{i,t}(\bmx)$ and $r(\bmx)$ are guaranteed to be Lipschitz continuous over $\mathcal{K}$, if we assume that the decision space $\mathcal{K}$ is compact.
\end{remark}

In this paper, we aim at developing \emph{mirror descent based} algorithms for solving problem (\ref{Online-CO}), which utilize the Bregman divergence as the distance-measuring function. Bregman divergences are a general class of distance-measuring functions, which include the standard Euclidean distance and Kullback-Leibler divergence as special cases. Moreover, mirror descent can generate adaptive updates to better reflect the geometry of the underlying constraint set, by carefully choosing the Bregman divergence; taking the unit simplex constraint set as an example, in this case the Bregman divergence is chosen as the Kullback-Leibler divergence (see, \eg, \cite{duchi2010colt,deming2018auto}).

Let $\omega: \reals^d \rightarrow \reals$ be a distance-measuring function, and define the \emph{Bregman divergence} associated with function $\omega$ as follows:
\BEAS
\breg_\omega(\bmx,\bmy) := \omega(\bmx) - \omega(\bmy) - \left< \nabla \omega(\bmy) , \bmx - \bmy \right>.
\EEAS

We make the following assumptions on the distance-measuring function $\omega$ and the associated Bregman divergence $\breg_\omega$.
\begin{assumption}\label{assump-fcn-phi}
Function $\omega$ is $\sigma_\omega$-strongly convex with respect to a norm $\| \cdot\|$ on the set $\calK$, that is,
\BEASN
\omega(\bmx) \geq \omega(\bmy) + \left< \grad\omega(\bmy), \bmx-\bmy \right> + \frac{\sigma_\omega}{2} \| \bmx - \bmy \|^2, \quad \forall \bmx,\bmy\in\calK
\EEASN
and has $G_\omega$-Lipschitz gradients on the set $\calK$,
\BEASN
\| \grad \omega(\bmx) - \grad \omega(\bmy) \|_\ast &\leq& G_\omega \| \bmx - \bmy \|  , \quad \forall \bmx,\bmy\in\calK.
\EEASN
\end{assumption}

\begin{assumption}\label{assump-breg}
The Bregman divergence $\breg_\omega$ is convex in its second argument $\bmy$ for every fixed $\bmx$, that is,
\BEASN
\breg_\omega \left(\bmx, \sum_{i=1}^d [\bm{\alpha}]_i \bmy_i \right)
&\leq&  \sum_{i=1}^d [\bm{\alpha}]_i \breg_\omega \left(\bmx, \bmy_i \right), \quad\forall \bm{\alpha}\in \Delta_d.
\EEASN
\end{assumption}

\begin{remark}
Assumptions 4 and 5 are standard in the literature on distributed mirror descent algorithms (see, \eg, \cite{xi2014arxiv,jadbabaie2018tac,deming2018auto}). As a simple example, they are satisfied when choosing the distance-measuring function as $\omega(\bmx) = \frac{1}{2} \|\bmx\|_2^2$, and the associated Bregman divergence is $\breg_\omega \left(\bmx,\bmy \right) = \frac{1}{2} \|\bmx - \bmy\|_2^2$.
\pr{
In addition, the Bregman divergence associated with the $\ell_p$ norm squared, \ie, $\omega(\bmx) = \frac{1}{2} \|\bmx\|_p^2$ with $p\in(1,2]$, satisfies the constraint in Assumption \ref{assump-breg}; another example is the Kullback-Leibler divergence that utilizes $\omega(\bm{x}) = \sum_{i=1}^d [\bm{x}]_i \ln ([\bm{x}]_i)$ as the distance-generating function (see \cite{bauschke2001}).
}
\end{remark}

\section{Full-Information Feedback Model}
This section focuses on solving online distributed composite optimization under full-information feedback. We first propose our algorithm and then establish its main convergence analysis results.
\subsection{Algorithm $\mathsf{ODCMD}$}
We now propose algorithm $\mathsf{ODCMD}$, which utilizes the Bregman divergence as distance-measuring functions, rather than the standard Euclidean distance employed by most approaches. The pseudo-code of the algorithm is presented in Algorithm \ref{alg-dcmd}.

\begin{algorithm}
\caption{$\mathsf{ODCMD}$: Online Distributed Composite Mirror Descent}
\label{alg-dcmd}
\begin{algorithmic}[1]
\REQUIRE a step size $\eta$, the optimization error sequence $\{\rho_t\}_{t=1}^{T}$
\ENSURE $\bmx_{i,1} = \arg\min_{\bmx\in\calK} \omega(\bmx), i\in[m]$
\FOR{$t=1$ to $T$}
\STATE
Node $i$ predicts $\bmx_{i,t} \in \calK$ and receives $\grad \ell_{i,t} (\bmx_{i,t})$
\STATE
Node $i$ computes a \emph{$\rho_t$-approximate} solution $\bmy_{i,t} \in \calK$ to the following optimization problem:
\BEASN
\mathop{\arg\min}_{\bmx\in\calK} \left< \grad_{i,t}, \bmx \right>  + r(\bmx) + \frac{1}{\eta} \breg_\omega (\bmx,\bmx_{i,t} )
\EEASN
where $\grad_{i,t} = \grad \ell_{i,t} (\bmx_{i,t})$
\STATE
Node $i$ updates $\bmx_{i,t+1}$ by communicating with its instant neighbors
\BEASN
\bmx_{i,t+1} &=& \sum_{j=1}^{m} [\wm(t)]_{ij} \bmy_{j,t}  \label{alg-ocmd-3}
\EEASN
\ENDFOR
\end{algorithmic}
\end{algorithm}

In step 3 in Algorithm \ref{alg-dcmd}, a $\rho_t$-approximate solution $\bmy_{i,t} \in \calK$ is computed in the following sense:
\BEASN
&&\left< \grad_{i,t}, \bmy_{i,t} \right> + r(\bmy_{i,t})+\frac{1}{\eta} \breg_\omega (\bmy_{i,t},\bmx_{i,t}) \nn\\
&\leq&  \left< \grad_{i,t} , \bmy_{i,t}^\star \right>   +  r(\bmy_{i,t}^\star) + \frac{1}{\eta} \breg_\omega (\bmy_{i,t}^\star,\bmx_{i,t})  + \rho_t
\EEASN
where
\BEASN
\bmy_{i,t}^\star &=& \mathop{\arg\min}_{\bmx\in\calK} \left< \grad_{i,t}, \bmx \right>  + r(\bmx) + \frac{1}{\eta} \breg_\omega (\bmx,\bmx_{i,t} ).
\EEASN

Note that the optimization problem arising at step 3 of $\mathsf{ODCMD}$ is only required to be solved \emph{approximately} (up to an additive error $\rho_t$). This is commonly observed in iterative optimization problems, since in general they cannot be solved to infinite precision. Moreover, the approximate computation of the optimization problem induces a sequence of errors that complicates the average regularized regret rate analysis of the algorithm.

\subsection{Main Convergence Results}
We first establish a theorem characterizing the main convergence results of Algorithm \ref{alg-dcmd}.
\begin{theorem}\label{corollary-main-cvx}
Let Assumptions 1--5 hold, and \pr{the vectors $\{  \bmx_{i,t} \}_{t=1}^{T} $ be generated by Algorithm \ref{alg-dcmd}}. Then, for all $T\geq 1$ and any $j\in[m]$,
\dy{
\BEASN
\overline{\reg}_j(T) &\leq& \frac{\mathsf{A}_0}{T} + \mathsf{A}_1 \frac{1}{\eta T} + \mathsf{A}_2 \eta + \frac{\mathsf{A}_3}{T} \sum_{t=1}^{T} \sqrt{\eta \rho_t} \nn\\
&& + \frac{\mathsf{A}_4}{T} \sum_{t=1}^{T} \sqrt{\frac{\rho_t}{\eta}}
\EEASN
}
where
\BEASN
\mathsf{A}_0 &=&  \frac{ 2 \vartheta }{1-\kappa} \left( G_\ell + G_r \right) \left(\sum_{i=1}^{m} \| \bmx_{i,1} \| \right) \nn\\
\mathsf{A}_1 &=&  \sum_{i=1}^{m} \breg_\omega (\bmx^\star , \bmx_{i,1}) \nn\\
\mathsf{A}_2 &=&  \frac{m}{\sigma_\omega} \left(  \frac{1}{2} G_\ell^2 + G_r (G_\ell + G_r) +  \frac{2 \vartheta}{1-\kappa} (G_\ell + G_r)^2 \right) \nn\\
\mathsf{A}_3 &=& m \sqrt{\frac{2}{\sigma_\omega}}  \left(  G_\ell +   \frac{ 2 \vartheta }{1-\kappa} \left( G_\ell + G_r \right) \right)  \nn\\
\mathsf{A}_4 &=& 2 m \sqrt{\frac{2}{\sigma_\omega}} G_\omega D_\calK
\EEASN
with $\vartheta = \big(1-\frac{\zeta}{4m^2}\big)^{-2}$ and $\kappa = \big(1-\frac{\zeta}{4m^2}\big)^{\frac{1}{B}}$.
\end{theorem}

It can be seen from Theorem \ref{corollary-main-cvx} that the average regularized regret of Algorithm \ref{alg-dcmd} relies on the properties of the parameters $\eta$ and $\rho_t$. In particular, we have the following corollary that characterizes the average regularized regret for every node in terms of the total number of rounds $T$, under specific choices of $\eta$ and $\rho_t$.
\begin{corollary}\label{corollary-main-cvx-2}
Under the conditions of Theorem \ref{theorem-basic-conv}, and taking
\BEASN
\eta =  \frac{1}{\sqrt{T}}, \qquad \qquad \rho_t = \bigo\left( \frac{1}{t^{3/2}} \right),
\quad t\in[T]
\EEASN
we have that, for all $T\geq 1$ and any $j\in[m]$,
\dy{
\BEASN
\overline{\reg}_j(T) &=& \bigo\left( 1/\sqrt{T} \right).
\EEASN
}
\end{corollary}
\begin{proof}
The desired result follows by combining the results in Theorem \ref{corollary-main-cvx}, the specific choices of $\eta$ and $\rho_t$, and the following inequality:
\BEASN
\sum_{t=1}^{T} \sqrt{\frac{1}{t^{3/2}}} &=& 1 + \sum_{t=2}^{T} \frac{1}{t^{3/4} } \nn\\
&\leq&  1 + \int_{u=1}^{T} \frac{1}{u^{3/4} } \mathrm{d}u \nn\\
&\leq&  4 T^{1/4}.
\EEASN
Therefore, the proof is complete.
\end{proof}


\begin{remark}
It would be of interest to investigate the optimal choice of the step size $\eta$. To be specific, we consider the error-free case, that is, $\rho_t = 0$ for all $t\in[T]$. Suppose that
\BEASN
\eta =  \frac{c_\eta}{\sqrt{T}},\qquad\qquad c_\eta >0
\EEASN
and combine this with the bound in Theorem \ref{corollary-main-cvx}, it is easy to see that
\dy{
\BEASN
\overline{\reg}_j(T) &\leq& \frac{\mathsf{A}_0}{T} + \left(  \frac{\mathsf{A}_1}{c_\eta}  + \mathsf{A}_2 c_\eta \right) \frac{1}{\sqrt{T}}.
\EEASN
}
It follows from some simple algebra that the optimal choice of $c_\eta$ is $c_\eta^{\star} = \sqrt{\mathsf{A}_1/\mathsf{A}_2}$.
Moreover, the upper bound on average regularized regret in Theorem 1 depends on parameter $B$, which represents the connectivity of the underlying network topology. It is easy to see that the upper bound on average regularized regret gets smaller for a network with better connectivity, which corresponds to smaller $B$.
\end{remark}

\begin{remark}
To the best of our knowledge, our proposed algorithm $\mathsf{ODCMD}$ is the first algorithm that utilizes the composite mirror descent to solve online distributed composite optimization problem and establishes explicit average regret rate results. In particular, Corollary \ref{corollary-main-cvx} shows that the \dy{average regularized regret of $\mathsf{ODCMD}$ is $\bigo( 1/\sqrt{T} )$}, matching that of the previously known centralized algorithm \cite{duchi2010colt}. Different from the algorithms in \cite{shi2015tsp,chen2012allerton,zeng2017scl}, $\mathsf{ODCMD}$ is an online mirror descent based algorithm that utilizes the Bregman divergence as the distance-measuring function, instead of the Euclidean distances; moreover, we only require the loss function to be Lipschitz continuous. Different from the work \cite{jadbabaie2018tac}, the objective function of our problem is composite that includes a regularization function, and we only require the optimization step to be solved approximately. \dy{Furthermore, our proposed algorithm can be applied to a more generalized model of (\ref{Online-CO}) that allows different regularization functions at different nodes.}
It is interesting to consider what happens if the error sequence $\{ \rho_t \}_{t=1}^{T}$ is not summable. For instance, if the error $\rho_t$ decreases as $\bigo\left( \frac{1}{t} \right)$, then the \dy{average regularized regret in Corollary \ref{corollary-main-cvx-2} is $\bigo\left( 1/T^{1/4} \right)$}.
\dy{
In fact, we can achieve the same average regularized regret scaling by choosing fixed optimization error, provided that the error scales as $\mathcal{O} (1/T^{3/2})$.
}
\end{remark}

\subsection{\dy{Discussions}}
\dy{
In this subsection, we illustrate the advantages of using general distance-generating functions $\omega(\bm{x})$ over the standard Euclidean distance, by providing the following two case studies.
}

\dy{
{\it Online Distributed Entropic Descent:}
Let $\mathcal{K} = \Delta_d$ (the probability simplex in $\mathbb{R}^d$), $\rho_t = 0$ for all $t\in[T]$ and $r(\bm{x}) = 0$. The distance-generating function is chosen as $\omega(\bm{x}) = \sum_{i=1}^{d} [\bm{x}]_i \ln ([\bm{x}]_i)$ and the associated Bregman divergence is $V_{\omega}(\bm{x},\bm{y}) = \sum_{i=1}^{d} [\bm{x}]_i \ln \left( \frac{[\bm{x}]_i}{[\bm{y}]_i} \right)$. In this case we can write step 3 in $\mathsf{ODCMD}$ explicitly as follows:
\BEASN
\left[\bm{y}_{i,t}^{\star}\right]_s &=& \frac{ [\bm{x}_{i,t}]_s \exp \left( - \eta [\nabla_{i,t}]_s \right) }
{ \sum_{j=1}^{d}  [\bm{x}_{i,t}]_j \exp \left( - \eta [\nabla_{i,t}]_j \right) }  ,  \qquad\qquad s\in[d] .
\EEASN
It is worth noting that choosing the standard Euclidean norm as the distance-generating function in this case would yield no explicit solutions, and in fact, it involves computing the solution of $d$-dimensional nonlinear equation at step 3.
}

\dy{
{\it $p$-norm {\sf ODCMD} with $\ell_1$-regularization:}
Let $r(\bm{x}) = \lambda \|\bm{x}\|_1$ and $\rho_t = 0$ for all $t\in[T]$. For the case of standard gradient descent which utilizes Euclidean distance as distance-generating function, we have $\omega(\bm{x}) = \frac{1}{2} \| \bm{x} \|_2^2$ and the associated Bregman divergence is $V_{\omega}(\bm{x},\bm{y}) = \frac{1}{2} \| \bm{x} - \bm{y} \|_2^2$. Based on the argument in Remark 3 and $\bm{x}_{i,1} = \bm{0}$ for all $i\in\mathcal{V}$, it leads to the following average regret bound:
\BEAS
\overline{\bm{Reg}}_j(T) = \frac{ \mathcal{O} \left( G_{\ell,2}  \| \bm{x}^{\star} \|_2  \right) } { \sqrt{T} }
\label{regret-l2}
\EEAS
where $G_{\ell,2}$ is the uniform bound on $\nabla \ell_{i,t}$ with respect to the Euclidean norm. For the case of mirror descent, we consider distance-generating function $\omega(\bm{x})$ which is the $\ell_p$ norm squared, that is, $\omega(\bm{x}) = \frac{1}{2} \| \bm{x} \|_p^2$, $p\in(1,2]$; the associated Bregman divergence is
$
V_{\omega}(\bm{x},\bm{y}) = \frac{1}{2} \| \bm{x} \|_p^2 + \frac{1}{2} \| \bm{y} \|_p^2
- \big< \bm{x}, \nabla \frac{1}{2} \| \bm{y} \|_p^2  \big>
$
(see, e.g., \cite{gentile2003}). Based on $V_{\omega}(\bm{x},\bm{y})$ and the fact that $\frac{1}{2} \| \bm{x} \|_p^2$ is $(p-1)$-strongly convex with respect to the $\ell_p$-norm, we find that
\BEAS
\overline{\bm{Reg}}_j(T) = \frac{ \mathcal{O} \left( \frac{1}{p-1}   G_{\ell,q}  \| \bm{x}^{\star} \|_p  \right) } { \sqrt{T} }
\label{regret-lp}
\EEAS
where $G_{\ell,q}$ is the uniform bound on $\nabla\ell_{i,t}$ with respect to the $\ell_q$ norm with $q$ satisfying $\frac{1}{p} + \frac{1}{q} = 1$.
}

\dy{
Following the intuition in \cite{xiao2010jmlr}, we transform the average regret bound (\ref{regret-lp}) in terms of $\ell_{\infty}$ and $\ell_{1}$ norms. Let $q = \ln (d)$ with $d\geq e^2$ (i.e., $q\geq 2$) and assume that $\max_{t\in[T]} \max_{i\in[m]} \max_{j\in[d]} | [\nabla \ell_{i,t}]_j | \leq G_{\ell,\infty}$, which gives $\| \nabla \ell_{i,t} \|_q \leq G_{\ell,\infty} d^{1/q} = e G_{\ell,\infty}$, for any $i\in[m]$ and $t\in[T]$. In addition, we have $\| \bm{x}^{\star} \|_p \leq \| \bm{x}^{\star} \|_1 $ (because $q>1$). Combining the preceding bounds, $p = \frac{\ln(d)}{\ln(d)-1}$ (due to $q = \ln (d)$ and $\frac{1}{p} + \frac{1}{q} = 1$) and (\ref{regret-lp}), yields
\BEAS
\overline{\bm{Reg}}_j(T) = \frac{ \mathcal{O} \left( \ln(d)   G_{\ell,\infty}  \| \bm{x}^{\star} \|_1  \right) } { \sqrt{T} }.
\label{regret-lp-v2}
\EEAS
For the case of Euclidean norm, using $\| \nabla \ell_{i,t} \|_2 \leq \sqrt{d} G_{\ell,\infty}$ we can replace $G_{\ell,2}$ in (\ref{regret-l2}) with $\sqrt{d} G_{\ell,\infty}$,
\BEAS
\overline{\bm{Reg}}_j(T) = \frac{ \mathcal{O} ( \sqrt{d} G_{\ell,\infty}  \| \bm{x}^{\star} \|_2  )} { \sqrt{T} }.
\label{regret-l2-v2}
\EEAS
}

\dy{
For distributed learning problems in which the features are dense (i.e., $G_{\ell,2}$ is close to $\sqrt{d} G_{\ell,\infty}$) and $\bm{x}^{\star}$ is very sparse (i.e., $\bm{x}^{\star}$ has only $k \ll d$ non-zero elements, due to $\ell_1$-regularization), the ratio between the bound in (\ref{regret-l2-v2}) and the bound in (\ref{regret-lp-v2}) becomes
\BEASN
\frac{\sqrt{d} \| \bm{x}^{\star} \|_2}{\ln(d) \| \bm{x}^{\star} \|_1}
\geq \frac{\sqrt{d} }{\ln(d) \sqrt{k}} > 1
\EEASN
for large values of $d$, i.e., problems in high dimensions. This means that using $\ell_p$ norm as the distance-generating function in $\mathsf{ODCMD}$ can lead to better optimality for each $T$.
}

\subsection{Proof of Theorem \ref{corollary-main-cvx}}
This subsection focuses on the proof of Theorem \ref{corollary-main-cvx}, which relies on the following two crucial lemmas, \ie, Lemmas \ref{theorem-basic-conv} and \ref{lemma-disagreement}. The first lemma establishes the basic convergence results of Algorithm \ref{alg-dcmd}.
\begin{lemma}\label{theorem-basic-conv}
Let Assumptions 1--5 hold, and the vectors $\{  \bmx_{i,t} \}_{t=1}^{T} $ and $\{  \bmy_{i,t} \}_{t=1}^{T}$ be generated by Algorithm \ref{alg-dcmd}.
Then, for all $T \geq 1$ and $j\in[m]$, we have
\dy{
\BEASN
\overline{\reg}_{j}(T)
&\leq&  \frac{1}{\eta T} \sum_{i=1}^{m} \breg_\omega (\bmx^\star , \bmx_{i,1}) +  \frac{m G_\ell^2 }{2\sigma_\omega} \eta \nn\\
&&+ G_r  \frac{1}{T} \sum_{t=1}^{T} \sum_{i=1}^{m}  \| \bmy_{i,t}^\star -  \bmx_{i,t} \|  \nn\\
&&+ \left( G_\ell +  \frac{2 }{\eta} G_\omega D_\calK \right) \frac{1}{T} \sum_{t=1}^{T} \sum_{i=1}^{m} \| \bmy_{i,t} -  \bmy_{i,t}^\star \|  \nn\\
&&+ \left( G_\ell + G_r \right)  \frac{1}{T} \sum_{t=1}^{T} \sum_{i=1}^{m} \| \bmx_{i,t} -  \bmx_{j,t} \|.
\EEASN
}
\end{lemma}
\begin{proof}
See Appendix \ref{appen:theorem:basic:conv}.
\end{proof}
\begin{remark}
The error bound in Lemma \ref{theorem-basic-conv} consists of five terms: the first three terms are optimization error terms; the fourth term is the penalty incurred for solving the optimization problem in step 3 in Algorithm \ref{alg-dcmd} with an approximate solution; and the last term is an additional penalty incurred due to having different decisions of nodes in the network, which is the cost of aligning each node's decision with those of its neighbors.
\end{remark}

The following lemma aims at providing bounds on the last three terms in Lemma \ref{theorem-basic-conv}.
\begin{lemma}\label{lemma-disagreement}
Let Assumptions \ref{assump:network}, \ref{assump-fcn-obj}, and \ref{assump-fcn-phi} hold, and let the vectors $\{  \bmx_{i,t} \} $ and $\{  \bmy_{i,t} \}$ be generated by Algorithm \ref{alg-dcmd}. We have that, for any $i,j\in[m]$,
\BIT
\item[(a)]
$
\| \bmy_{i,t} -  \bmy_{i,t}^\star \|  \leq \sqrt{  \frac{2}{\sigma_\omega } \eta \rho_t } .
$
\item[(b)]
$
\| \bmy_{i,t}^\star - \bmx_{i,t} \| \leq \frac{1}{\sigma_\omega} \left( G_\ell + G_r \right) \eta.
$
\item[(c)] The disagreement among nodes satisfies
\BEASN
\hspace{-1.5em}\sum_{t=1}^{T} \sum_{i=1}^{m} \| \bmx_{i,t} - \bmx_{j,t} \|
&\leq& \frac{2 \vartheta}{1-\kappa} \left(\sum_{i=1}^{m} \| \bmx_{i,1} \| \right) \nn\\
&&\hspace{-10em}+   \frac{ 2 m \vartheta  }{\sigma_\omega (1-\kappa)} \left( G_\ell + G_r \right) \eta T
+ \frac{2 m \vartheta}{1-\kappa}  \sqrt{  \frac{2}{\sigma_\omega } } \sum_{t=1}^{T} \sqrt{\eta \rho_t }.
\EEASN
\EIT
\end{lemma}
\begin{proof}
See Appendix \ref{appen:lemma:disagreement}.
\end{proof}

Therefore, it is straightforward to derive the convergence results in Theorem \ref{corollary-main-cvx} by combining the results in Lemmas \ref{theorem-basic-conv} and \ref{lemma-disagreement}.



\section{Bandit Feedback Model}
\pr{
In this section, we focus on the case of bandit feedback, where at the end of each round, node only has access to the information of function values. The pseudo-code of our algorithm adapt to this feedback is provided in Algorithm \ref{alg-dcmd-bandit}.
}
\subsection{Algorithm $\mathsf{BanODCMD}$}
Under bandit feedback, each node $i$ can only observe the value of the loss function at (or near) point $\bmx_{i,t}$ at round $t$, instead of the entire loss function $\ell_{i,t}$. Specifically, each node does not know the gradient of $\ell_{i,t}$ at $\bmx_{i,t}$. To this end, we propose $\mathsf{BanODCMD}$, where at each round each node queries the loss function at two randomized points around $\bmx_{i,t}$, rather than the gradient $\grad \ell_{i,t} (\bmx_{i,t})$.

In the bandit setting, we impose the following standard assumption on the constraint set $\calK$ (see, \eg, \cite{flaxman2005,agarwal2010}), instead of Assumption \ref{assump-set-bounded}.
\begin{assumption}\label{assump-set-bounded-bandit}
The constraint set $\calK$ contains the Euclidean ball of radius $\underline{R}$ centered at the origin and is contained in the Euclidean ball of radius $\overline{R}$, that is,
$
\mathbb{B}_{\underline{R}}   \subseteq  \calK   \subseteq  \mathbb{B}_{\overline{R}}.
$
\end{assumption}

\begin{algorithm}
\caption{$\mathsf{BanODCMD}$: Bandit Online Distributed Composite Mirror Descent}
\label{alg-dcmd-bandit}
\begin{algorithmic}[1]
\REQUIRE step size $\eta$, the optimization error sequence $\{\rho_t\}_{t=1}^{T}$, the exploration parameter $\delta$, and the shrinkage parameter $\xi$
\ENSURE $\bmx_{i,1} = \arg\min_{\bmx\in(1-\xi)\calK} \omega(\bmx), i\in[m]$
\FOR{$t=1$ to $T$}
\STATE
Node $i$ queries $\ell_{i,t} (\bmx_{i,t} + \delta \bmu_{i,t})$ and $\ell_{i,t} (\bmx_{i,t} - \delta \bmu_{i,t})$, where $\bmu_{i,t}$ is a unit vector generated uniformly at random (\ie, $\|\bmu_{i,t}\|_2 = 1 $), and node $i$ sets
\BEASN
\grade_{i,t} = \frac{d}{2 \delta} \left(  \ell_{i,t} (\bmx_{i,t} + \delta \bmu_{i,t}) - \ell_{i,t} (\bmx_{i,t} - \delta \bmu_{i,t})\right)  \bmu_{i,t}
\EEASN
\STATE
Node $i$ computes a \emph{$\rho_t$-approximate} solution $\bmy_{i,t} \in (1-\xi)\calK$ to the following optimization problem:
\BEASN
\mathop{\arg\min}_{ \bmx\in(1-\xi)\calK } \left< \grade_{i,t}, \bmx \right>  + r(\bmx) + \frac{1}{\eta} \breg_\omega (\bmx,\bmx_{i,t} )
\EEASN
where $(1-\xi)\calK = \{ (1-\xi) \bmx \mid \bmx\in\mathcal{K} \}$
\STATE
Node $i$ updates $\bmx_{i,t+1}$ by communicating with its instant neighbors
\BEASN
\bmx_{i,t+1} &=& \sum_{j=1}^{m} [\wm(t)]_{ij} \bmy_{j,t}  \label{alg-ocmd-3}
\EEASN
\ENDFOR
\end{algorithmic}
\end{algorithm}

As in the full-information feedback case, a $\rho_t$-approximate solution $\bmy_{i,t} \in (1-\xi)\calK$ in step 3 in Algorithm \ref{alg-dcmd-bandit} is computed in the following sense:
\BEASN
&&\left< \grade_{i,t}, \bmy_{i,t} \right> + r(\bmy_{i,t})+\frac{1}{\eta} \breg_\omega (\bmy_{i,t},\bmx_{i,t}) \nn\\
&\leq&  \left< \grade_{i,t} , \bmy_{i,t}^\star \right>   +  r(\bmy_{i,t}^\star) + \frac{1}{\eta} \breg_\omega (\bmy_{i,t}^\star,\bmx_{i,t})  + \rho_t
\EEASN
where
\BEASN
\bmy_{i,t}^\star &=& \mathop{\arg\min}_{ \bmx\in(1-\xi)\calK } \left< \grade_{i,t}, \bmx \right>  + r(\bmx) + \frac{1}{\eta} \breg_\omega (\bmx,\bmx_{i,t} ).
\EEASN

\begin{remark}
In the bandit feedback model, new challenges arise in the absence of gradient information when designing the algorithm; we remedy this by introducing a distributed gradient estimator that is based on two-point bandit feedback from the loss function. Specifically, in Algorithm \ref{alg-dcmd-bandit}, only two functional evaluations are utilized to construct a gradient estimator $\grade_{i,t}$, which is different from Algorithm \ref{alg-dcmd} where the gradient information is required.
\end{remark}

The following two lemmas characterize the basic properties of the gradient estimator $\grade_{i,t}$, which play a crucial role in the average regularized regret analysis of Algorithm \ref{alg-dcmd-bandit}. Define a smoothed function of the following form:
\BEASN
\breve{\ell}_{i,t} (\bmx) := \expect_{\bm{v}\in\mathbb{B}} \left[\ell_{i,t} (\bmx + \delta \bm{v})\right]
\EEASN
where $\bm{v}$ is a vector selected uniformly at random from the unit ball $\mathbb{B}$ in $\reals^{d}$. Then we have the following lemma, whose proof is quite straightforward (see, for example, \cite{agarwal2010,shamir2017}).
\begin{lemma}\label{lemma-bandit}
Let $\mathcal{F}_t$ be the $\sigma$-field generated by the entire history of the random variables to round $t$, we have
\BIT
\item[(a)]
$\left| \breve{\ell}_{i,t} (\bmx) - \breve{\ell}_{i,t} (\bmy) \right| \leq G_\ell \| \bmx - \bmy \|, \qquad \forall \bmx,\bmy\in\calK$.
\item[(b)]
$\max_{\bmx \in \calK} \left| \breve{\ell}_{i,t} (\bmx) - \ell_{i,t} (\bmx) \right| \leq  G_\ell \delta  $.
\item[(c)]
$\expect \left[\grade_{i,t} \mid \mathcal{F}_t\right] = \grad \breve{\ell}_{i,t} (\bmx_{i,t})$.
\EIT
\end{lemma}

\begin{lemma}\label{lemma-bandit-2}
\pr{
The gradient estimator $\grade_{i,t}$ satisfies the following for all $i\in [m]$ and $t\in [T]$:
\[
\| \grade_{i,t}\|_\ast  \leq \overline{p}\, \overline{p}_\ast d G_\ell
\]
where $\overline{p}$ and $\overline{p}_\ast$ are constants that satisfy the inequalities $\| \bmx \| \leq \overline{p}  \| \bmx \|_2$ and $\| \bmx \|_{\ast} \leq \overline{p}_\ast  \| \bmx \|_2$ for any $\bmx$, respectively.
}
\end{lemma}
\begin{proof}
It follows from the explicit expression for $\grade_{i,t}$ that
\BEASN
\| \grade_{i,t} \|_\ast &\!\!\!\!=\!\!\!\!& \left\| \frac{d}{2 \delta} \left(  \ell_{i,t} (\bmx_{i,t} \!+\! \delta \bmu_{i,t}) \!-\! \ell_{i,t} (\bmx_{i,t} \!-\! \delta \bmu_{i,t})\right)  \bmu_{i,t} \right\|_\ast \nn\\
&\!\!\!\!\leq\!\!\!\!& \frac{d}{2 \delta} \left| \ell_{i,t} (\bmx_{i,t} \!+\! \delta \bmu_{i,t}) \!-\! \ell_{i,t} (\bmx_{i,t} \!-\! \delta \bmu_{i,t})  \right| \cdot \left\| \bmu_{i,t} \right\|_\ast \nn\\
&\!\!\!\!\leq\!\!\!\!& d G_\ell \| \bmu_{i,t} \| \| \bmu_{i,t} \|_\ast \nn\\
&\!\!\!\!\leq\!\!\!\!&  \overline{p}\, \overline{p}_\ast d G_\ell
\EEASN
where in the second inequality we used the Lipschitz continuity of function $\ell_{i,t}$ (cf. Assumption \ref{assump-fcn-obj}), and in the last inequality we used the fact that $\| \bmu_{i,t} \|_2 = 1$ and the equivalence of the norms on finite-dimensional real vector space, \ie, $  \| \bmx \| \leq \overline{p} \| \bmx \|_2 $ and $  \| \bmx \|_\ast \leq \overline{p}_\ast  \| \bmx \|_2 $. The proof is complete.
\end{proof}

\begin{remark}
In step 3 the minimizer $\bmy_{i,t}^\star$ as well as its approximate solution $\bmy_{i,t}$ are required to belong to the set $(1-\xi)\calK$, which is designed to guarantee that the query points $\bmx_{i,t} \pm \delta \bmu_{i,t}$ belong to the constraint set $\calK$ (see Lemma \ref{lemma-flaxman} in the sequel). This is a common technique used widely in online bandit convex optimization (see, \eg, \cite{flaxman2005,agarwal2010}). Note also that the convergence analysis of Algorithm \ref{alg-dcmd-bandit} is much more involved than that of Algorithm \ref{alg-dcmd}, since there exist two new parameters $\delta$ and $\xi$ in Algorithm \ref{alg-dcmd-bandit}, due to the use of the gradient estimator.
\end{remark}

The following lemma is used to design the parameters in Algorithm \ref{alg-dcmd-bandit}, in order to guarantee that the decisions $\bmx_{i,t} \pm \delta \bmu_{i,t}$ belong to the constraint set $\calK$.
\begin{lemma}[\cite{flaxman2005}, Observation 2]\label{lemma-flaxman}
For any $\bmx \in (1-\xi) \calK$ and any unit vector $\bmu$, it holds that $\bmx + \delta \bmu \in \calK$ for any $\delta \in [0, \xi \underline{R}]$.
\end{lemma}

\subsection{Main Convergence Results}
We now establish a theorem that characterizes the average regularized regret of Algorithm \ref{alg-dcmd-bandit}.
\begin{theorem}\label{theorem-bandit}
Let Assumptions 1, 3, 4, 5, and 6 hold. Let the decision sequences $\{  \bmx_{i,t} \}_{t=1}^{T} $ and $\{  \bmy_{i,t} \}_{t=1}^{T}$ be generated by Algorithm \ref{alg-dcmd-bandit}. Set $\delta \leq  \xi \underline{R} $.
Then, for all $T \geq 1$ and $j\in[m]$, we have
\dy{
\BEASN
\expect \left[  \overline{\reg}_j(T) \right]
&\leq& \frac{\mathsf{B}_0}{T}  + \mathsf{B}_1 \frac{1}{\eta T} +  \mathsf{B}_2 \eta  + \frac{\mathsf{B}_3}{T} \sum_{t=1}^{T} \sqrt{\eta \rho_t}  \nn\\
&&+ \frac{\mathsf{B}_4}{T}  \sum_{t=1}^{T} \sqrt{\frac{\rho_t}{\eta}} + \mathsf{B}_5   \delta  + \mathsf{B}_6 \xi
\EEASN
}
where
\BEASN
\mathsf{B}_0 &=& \frac{2 \vartheta}{1-\kappa} \left( G_\ell + G_r \right)  \left(\sum_{i=1}^{m} \expect [ \| \bmx_{i,1} \|] \right) \nn\\
\mathsf{B}_1 &=& \sum_{i=1}^m  \expect \left[ \breg_\omega ((1-\xi)\bmx^\star, \bmx_{i,1}) \right] \nn\\
\mathsf{B}_2 &=& \frac{m}{\sigma_\omega} \left(  \frac{1}{2} (\overline{p}\,\overline{p}_\ast)^2 d^2 G_\ell^2 +  G_r  \left( \overline{p}\,\overline{p}_\ast d G_\ell + G_r \right) \right. \nn\\
&&\left. +  \frac{ 2 \vartheta  }{1-\kappa} \left( G_\ell + G_r \right) \left( \overline{p}\,\overline{p}_\ast d G_\ell + G_r \right)  \right)   \nn\\
\mathsf{B}_3 &=& m \sqrt{  \frac{2}{\sigma_\omega } } \left( \overline{p}\,\overline{p}_\ast d G_\ell  +  \frac{ 2 \vartheta  }{1-\kappa} \left( G_\ell + G_r \right) \right)  \nn\\
\mathsf{B}_4 &=& 4 m \sqrt{  \frac{2}{\sigma_\omega } } \overline{p} G_\omega \overline{R} \nn\\
\mathsf{B}_5 &=& 2 m G_\ell  \nn\\
\mathsf{B}_6 &=& m \overline{p} \left( G_\ell + G_r \right) \overline{R}.
\EEASN
\end{theorem}


Compared to the average regularized regret bound in Theorem \ref{corollary-main-cvx}, there have two additional terms in the average regularized regret bound in Theorem \ref{theorem-bandit}, which are introduced by using the gradient estimator instead of the gradient in the algorithm. In the sequel, we will see that the \dy{$\mathcal{O}( 1/\sqrt{T} )$ average regularized regret} can be recovered by choosing the parameters $\delta$ and $\xi$ appropriately.

Note that term $\mathsf{B}_1$ can be bounded by using the fact that function $\omega$ has $G_\omega$-gradient over the set $\calK$, that is,
\BEASN
\mathsf{B}_1 &\leq& \sum_{i=1}^m   \frac{G_\omega}{2}  \left\| (1-\xi)\bmx^\star -  \bmx_{i,1} \right\|^2 \\
&\leq& G_\omega  \sum_{i=1}^m   \left( \left\| (1-\xi)\bmx^\star  \right\|^2 + \left\| \bmx_{i,1} \right\|^2 \right) \\
&\leq& 2 m \overline{p}^2 G_\omega  \overline{R}^2.
\EEASN

Similarly, we have the following corollary that characterizes the average regularized regret for every node in terms of the total number of rounds $T$, by combining the results in Theorem \ref{theorem-bandit} and the preceding inequality.
\begin{corollary}\label{corollary-bandit}
\pr{
Under the conditions of Theorem \ref{theorem-bandit}, and taking
\BEASN
\eta &=&  \frac{1}{\overline{p}\,\overline{p}_\ast d \sqrt{T}} , \hspace{3em}  \rho_t = \bigo\left( \frac{1}{t^{3/2}} \right)  \nn\\
\delta &=& \frac{1}{\sqrt{T}}, \hspace{5.5em} \xi = \frac{\delta}{\underline{R}}, \qquad t\in[T]
\EEASN
we have that, for all $T\geq 1$ and any $j\in[m]$,
\BEASN
\overline{\reg}_j(T) &=&  \bigo\left( \frac{\overline{p}\,\overline{p}_\ast d}{\sqrt{T}} \right) .
\EEASN
}
\end{corollary}

\begin{remark}
In fact, we can construct a different gradient estimator by using only one single functional evaluation \cite{flaxman2005} or multiple functional evaluations \cite{agarwal2010}. However, building on our convergence analysis results it is easy to prove that algorithm with one-point bandit feedback would exhibit worse \dy{average regret guarantees (\eg, scaling as $1/T^{1/6}$)} and algorithm with multi-point bandit feedback would exhibit the same average regret scaling as that of {\sf BanODCMD}. Moreover, algorithm {\sf BanODCMD} has a natural connection with the distributed zeroth-order algorithms in \cite{yuan2015tnnls}, in the following sense: i) Both algorithms rely on two functional evaluations at every iteration; ii) {\sf BanODCMD} can tackle distributed composite optimization in online setting, while \cite{yuan2015tnnls} deals with off-line distributed optimization where the objective functions of nodes are fixed; and iii) The key difference is that in algorithm {\sf BanODCMD} the query points (\ie, $\bmx_{i,t} \pm \delta \bmu_{i,t}$) must lie in the decision space $\mathcal{K}$, which is guaranteed by projecting the estimates onto the shrunk set $(1-\xi)\mathcal{K}$ and choosing the exploration parameter $\delta$ and the shrinkage parameter $\xi$ appropriately. In contrast, such a requirement is not needed in distributed zeroth-order algorithms \cite{yuan2015tnnls}.
\end{remark}

\begin{remark}
\pr{
Compared to the algorithms \cite{hosseini2016tac,jadbabaie2018tac,lee2017tac,yan2013tkde} that are under full-information feedback, algorithm $\mathsf{BanODCMD}$ is the first online distributed algorithm under bandit feedback that only requires the information of functional values at two queried points. In addition, our $\mathsf{BanODCMD}$ algorithm can deal with distributed optimization problem with composite structure, while in \cite{hosseini2016tac,jadbabaie2018tac,lee2017tac,yan2013tkde} no regularization functions are considered in their objective functions.
}
\end{remark}


\begin{remark}
\pr{
It is worth noting that the average regularized regret of algorithm $\mathsf{BanODCMD}$ scales with the problem dimension as $\overline{p}\,\overline{p}_{\ast} d$, where $\overline{p}$ and $\overline{p}_{\ast}$ depend on the norm used. For example, in the case of Euclidean norm $\| \cdot \|_2$, we have $\overline{p}\,\overline{p}_{\ast}=1$ and the average regularized regret scales with the dimension as $\mathcal{O}(d)$. This dependence is identical to that of \cite{flaxman2005}, which considers centralized online bandit optimization. In the case of $\ell_p$ norm $\|\cdot\|_p$ with $p\geq 1$, it is easy to show that $\overline{p}\,\overline{p}_{\ast} \leq \sqrt{d}$, which results in an average regularized regret scaling of $\mathcal{O}(d\sqrt{d}/\sqrt{T})$. This scaling is better than the $\bigo( d^2 / \sqrt{T} )$ average regret scaling in \cite{agarwal2010}, but worse than that of \cite{shamir2017}, where the average regret scales optimally with dimension as $\sqrt{d}$; however, it is worth noting that the problem considered in \cite{agarwal2010,shamir2017} is in centralized setting and without composite structure.
}
\end{remark}

\subsection{Proof of Theorem \ref{theorem-bandit}}
We establish the proof of Theorem \ref{theorem-bandit}, by using the results in Lemmas \ref{lemma-bandit}, \ref{lemma-bandit-2} and \ref{lemma-flaxman}. We begin our proof by first showing that the decisions $\bmx_{i,t} \pm \delta \bmu_{i,t}$ belong to the constraint set $\calK$. It follows from step 3 in Algorithm \ref{alg-dcmd-bandit} that $\bmy_{i,t} \in (1-\xi) \calK$, which implies that $\bmx_{i,t+1} \in (1-\xi) \calK$. This, combined with Lemma \ref{lemma-flaxman}, gives the condition $\delta \leq \xi \underline{R}$ that ensures $\bmx_{i,t} \pm \delta \bmu_{i,t} \in \calK$.
We now prove the main results, by following an argument similar to that of Lemma \ref{theorem-basic-conv}, and we immediately have that for any $\bmx\in(1-\xi)\calK$,
\BEASN
\hspace{-0.5em}&&\left<  \grade_{i,t} + \grad r(\bmy_{i,t}^\star)  ,  \bmy_{i,t}^\star -   \bmx  \right>   +  \left<  \grade_{i,t} , \bmx_{i,t} -  \bmy_{i,t} \right>  \nn\\
\hspace{-0.5em}&\leq&  \frac{1}{\eta}  \left( \breg_\omega (\bmx , \bmx_{i,t})  -  \breg_\omega (\bmy_{i,t}, \bmx_{i,t})  - \breg_\omega (\bmx , \bmy_{i,t}) \right) \nn\\
\hspace{-0.5em}&&+ \frac{4 \overline{p} G_\omega \overline{R}}{\eta}  \| \bmy_{i,t} -  \bmy_{i,t}^\star \| \!+\! \frac{\sigma_\omega}{2\eta} \| \bmx_{i,t} -  \bmy_{i,t} \|^2 \!+\! \frac{\eta}{2\sigma_\omega} \|  \grade_{i,t} \|^2_\ast
\label{theorem-badit-1}
\EEASN
where we used the new bound $2 \overline{p} \overline{R}$ on the diameter of $\calK$, because of $\|  \bmx - \bmy \| \leq \| \bmx \| + \| \bmy \| \leq \overline{p} \left( \| \bmx \|_2 + \| \bmy \|_2 \right)  \leq 2 \overline{p} \overline{R} $. By adding and subtracting $\bmy_{i,t}^{\star}$ to the second term on the left-hand side and then taking the conditional expectation, we have for any $\bmx\in(1-\xi)\calK$,
\BEAS
&& \expect \left[ \left<  \grade_{i,t} + \grad r(\bmy_{i,t}^\star)  ,  \bmy_{i,t}^\star -   \bmx  \right> | \mathcal{F}_t\right] \nn\\
&&+  \expect \left[ \left<  \grade_{i,t} , \bmx_{i,t} -  \bmy_{i,t}^{\star} \right>  \mid \mathcal{F}_t\right] \nn\\
&&+ \expect \left[ \left<  \grade_{i,t} , \bmy_{i,t}^{\star} - \bmy_{i,t}  \right>  \mid \mathcal{F}_t\right] \nn\\
&=& \expect \left[ \left<  \grade_{i,t} ,  \bmx_{i,t} -   \bmx  \right> \mid \mathcal{F}_t\right] \nn\\
&&+ \expect \left[ \left<\grad r(\bmy_{i,t}^\star)  ,  \bmy_{i,t}^\star -   \bmx  \right> \mid \mathcal{F}_t\right] \nn\\
&&+ \expect \left[ \left<  \grade_{i,t} , \bmy_{i,t}^{\star} - \bmy_{i,t}  \right>  \mid \mathcal{F}_t\right]  \nn\\
&\geq& \breve{\ell}_{i,t}(\bmx_{i,t}) - \breve{\ell}_{i,t}(\bmx) + r(\bmy_{i,t}^\star) - r(\bmx) \nn\\
&&- \expect \left[ \| \grade_{i,t}\|_\ast \cdot \| \bmy_{i,t} -  \bmy_{i,t}^\star \| \mid \mathcal{F}_t\right] \nn\\
&\geq& \breve{\ell}_{i,t}(\bmx_{i,t}) - \breve{\ell}_{i,t}(\bmx) + r(\bmy_{i,t}^\star) - r(\bmx) \nn\\
&&- \overline{p}\, \overline{p}_\ast d G_\ell \| \bmy_{i,t} -  \bmy_{i,t}^\star \|
\label{theorem-badit-2}
\EEAS
where the first and second inequalities follow respectively from Lemma \ref{lemma-bandit}(c) and Lemma \ref{lemma-bandit-2}. Doing the same operation on the right-hand side and then combining with inequality (\ref{theorem-badit-2}), taking the total expectation, and summing over $i=1,\dots,m$, yields
\BEAS
&& \sum_{i=1}^m \expect \left[ \breve{\ell}_{i,t}(\bmx_{i,t})  + r(\bmy_{i,t}^\star) \right] \nn\\
&&- \sum_{i=1}^m\expect \left[ \breve{\ell}_{i,t}(\bmx) + r(\bmx)  \right] \nn\\
&\leq&  \frac{1}{\eta}  \sum_{i=1}^m \left( \breg_\omega (\bmx , \bmx_{i,t})  - \breg_\omega (\bmx , \bmy_{i,t})  \right) \nn\\
&&+ \left( \overline{p}\,\overline{p}_\ast d G_\ell + 4 \overline{p} G_\omega \overline{R} \frac{1}{\eta} \right)  \sum_{i=1}^m \| \bmy_{i,t} -  \bmy_{i,t}^\star \|   \nn\\
&& +   \frac{m (\overline{p}\,\overline{p}_\ast)^2 d^2 G_\ell^2}{2 \sigma_\omega} \eta
\label{theorem-badit-3}
\EEAS
where we used the same reasoning as that of (\ref{lemma-basic-conv-6}) (see Appendix \ref{appen:theorem:basic:conv}) and the bound on $\|  \grade_{i,t} \|_\ast$ (cf. Lemma \ref{lemma-bandit-2}).

Following similar lines as that of (\ref{lemma-basic-conv-7a}), (\ref{lemma-basic-conv-7b}), (\ref{theorem-main-cvx-2}) and (\ref{theorem-main-cvx-3}) (see Appendix \ref{appen:theorem:basic:conv}), we further have for any $\bmx\in\calK$,
\BEAS
&& \sum_{t=1}^T \sum_{i=1}^m \expect \left[ \breve{\ell}_{i,t}(\bmx_{j,t}) + r(\bmx_{j,t}) \right] \nn\\
&&- \sum_{t=1}^T \sum_{i=1}^m\expect \left[ \breve{\ell}_{i,t}\left( (1-\xi)\bmx \right) + r\left( (1-\xi)\bmx \right)  \right] \nn\\
&\leq&  \frac{1}{\eta} \sum_{i=1}^m  \expect \left[ \breg_\omega ((1-\xi)\bmx , \bmx_{i,1}) \right] +  \frac{m (\overline{p}\,\overline{p}_\ast)^2 d^2 G_\ell^2 }{2\sigma_\omega} \eta T \nn\\
&&+ \left( \overline{p}\,\overline{p}_\ast d G_\ell + 4 \overline{p} G_\omega \overline{R} \frac{1}{\eta} \right)  \sum_{t=1}^T \sum_{i=1}^m \expect \left[  \| \bmy_{i,t} -  \bmy_{i,t}^\star \| \right] \nn\\
&&+G_r  \sum_{t=1}^{T} \sum_{i=1}^{m}  \expect \left[  \| \bmy_{i,t}^\star -  \bmx_{i,t} \| \right]  \nn\\
&&+ \left(  G_\ell + G_r \right)  \sum_{t=1}^{T} \sum_{i=1}^{m} \expect \left[  \| \bmx_{i,t} -  \bmx_{j,t} \| \right]
\label{theorem-badit-4}
\EEAS
because function $\breve{\ell}_{i,t}$ is $G_\ell$-Lipschitz continuous. We now relate the left-hand side and the expected average regularized regret associated with the loss functions $\ell_{i,t}$, by Lemma \ref{lemma-bandit}(b),
\BEAS
- G_\ell \delta
&\leq&  \breve{\ell}_{i,t}(\bmx_{j,t}) - \ell_{i,t}(\bmx_{j,t})    \nn\\
- G_\ell \delta
&\leq&   \ell_{i,t}\left( (1-\xi)\bmx \right)  -  \breve{\ell}_{i,t}\left( (1-\xi)\bmx \right).
\label{theorem-badit-5}
\EEAS
We use the Lipschitz continuity of functions $\ell_{i,t}$ and $r$ to further obtain that for any $\bmx\in\calK$,
\BEAS
\ell_{i,t}\left( (1-\xi)\bmx \right)  -  \ell_{i,t}\left( \bmx \right)
&\leq& G_\ell \xi \| \bmx \| \leq \overline{p} G_\ell   \overline{R} \xi \nn\\
r\left( (1-\xi)\bmx \right) - r\left( \bmx \right)
&\leq& G_r \xi \| \bmx \| \leq \overline{p} G_r   \overline{R} \xi.
\label{theorem-badit-6}
\EEAS
Combining the results in inequalities (\ref{theorem-badit-4}), (\ref{theorem-badit-5}) and (\ref{theorem-badit-6}), and setting $\bmx = \bmx^\star$, we have
\dy{
\BEAS
&& \expect \left[  \overline{\reg}_j(T) \right]  \nn\\
&\leq&  \frac{1}{\eta T} \sum_{i=1}^m  \expect \left[ \breg_\omega ((1-\xi)\bmx^\star, \bmx_{i,1}) \right] +  \frac{m (\overline{p}\,\overline{p}_\ast)^2 d^2 G_\ell^2 }{2\sigma_\omega} \eta  \nn\\
&&+  2 m G_\ell \frac{1}{T} \sum_{t=1}^T  \delta +  m \overline{p} \left( G_\ell + G_r \right) \overline{R} \xi   \nn\\
&&+ \left( \overline{p}\,\overline{p}_\ast d G_\ell + 4 \overline{p} G_\omega \overline{R} \frac{1}{\eta} \right) \frac{1}{T}  \sum_{t=1}^T \sum_{i=1}^m \expect \left[  \| \bmy_{i,t} -  \bmy_{i,t}^\star \| \right] \nn\\
&&+G_r  \frac{1}{T} \sum_{t=1}^{T} \sum_{i=1}^{m}  \expect \left[  \| \bmy_{i,t}^\star -  \bmx_{i,t} \| \right]  \nn\\
&&+ \left(  G_\ell + G_r \right)  \frac{1}{T} \sum_{t=1}^{T} \sum_{i=1}^{m} \expect \left[  \| \bmx_{i,t} -  \bmx_{j,t} \| \right].
\label{theorem-badit-7}
\EEAS
}

Hence, we are left to bound the last three terms on the right-hand side of (\ref{theorem-badit-7}). In fact, they can be respectively bounded by using the results in Lemma \ref{lemma-disagreement}, and we can derive the desired bound in Theorem \ref{theorem-bandit}.
$\hfill\blacksquare$

\section{Simulation Results}
In this section, we consider a distributed online regularized linear regression problem:
\BEQ
\begin{array}{lll}
\mathrm{min.}    &  & \sum\limits_{t=1}^{T}\sum\limits_{i=1}^{m} \left( \frac{1}{2} \left( \left< \bm{b}_{i,t},\bmx \right> - y_{i,t} \right)^2 + \frac{\lambda_{1}}{2} \| \bmx \|_2^2 + \lambda_{2} \| \bmx \|_1 \right) \\
\mathrm{s.t.}  &  & \bmx \in \calK
\end{array}
\label{ridge-reg}
\EEQ
where the data sequence $ \{ (\bm{b}_{i,t} , y_{i,t}) \}_{t=1}^{T} $ is known only to node $i$, and every entry of the input vector $\bm{b}_{i,t}$ was generated uniformly from the interval $(-1,1)$ and the response is given by
$$y_{i,t} = \left< \bm{b}_{i,t},\bmx_{0} \right> + \varepsilon $$
where $[\bmx_{0}]_i = 1$ for $1 \leq i \leq \left\lfloor \frac{d}{2} \right\rfloor$ and $0$ otherwise and the noise $\varepsilon$ was generated independent and identically distributed from the normal distribution $N(0,1)$. We set $\calK = \{  \bmx\in\reals^d \mid \| \bmx \|_2 \leq \overline{R} \}$, and set the distance-measuring function as $\omega(\bmx) = \frac{1}{2} \| \bmx \|_2^2$. In this case the optimization problem $\bmy_{i,t}^\star = \mathop{\arg\min}_{\bmx\in\calK} \left< \grad_{i,t}, \bmx \right>  + r(\bmx) + \frac{1}{\eta} \breg_\omega (\bmx,\bmx_{i,t} )$ in $\mathsf{ODCMD}$ can be solved as follows:
\BEASN
\widehat{\bmy}_{i,t} &=&  \bmx_{i,t} - \eta\grad_{i,t} \nn\\
\left[ \widetilde{\bmy}_{i,t} \right]_s &=& \mathrm{sign}\left( \left[ \widehat{\bmy}_{i,t} \right]_s \right) \left[ \left| \left[ \widehat{\bmy}_{i,t} \right]_s \right| - \eta \lambda_2 \right]_+, \quad  s\in[d]\nn\\
\bmy_{i,t}^\star &=& \frac{\overline{R}}{\max\left( \| \widetilde{\bmy}_{i,t} \|_2 , \overline{R}  \right)} \widetilde{\bmy}_{i,t}
\EEASN
where $\grad_{i,t} =  \left( \left< \bm{b}_{i,t},\bmx_{i,t} \right> - y_{i,t} \right) \bm{b}_{i,t} + \lambda_{1}  \bmx_{i,t} $. Similarly, the optimization problem in $\mathsf{BanODCMD}$ (\ie, step 3) can be solved as well by replacing $\grad_{i,t}$ and $\overline{R}$ with $\grade_{i,t}$ and $(1-\xi)\overline{R}$ in the preceding equations, respectively.

Set $d = 10$, $\overline{R}=1$, $\lambda_1 = 1$, $\lambda_2 = 0.1$, and \pr{$\eta = \frac{1}{d \sqrt{T}}$}, and we first implement the algorithms over a randomly generated network of $m = 30$ nodes shown in Fig. 1. The network is changing according to the following way: at round $2t-1$ ($t\in[1,\lceil T/2 \rceil]$), half of the edges in the network (see Fig. 1) are activated randomly; at round $2t$ ($t\in[1,\lfloor T/2 \rfloor]$), the other half edges are activated. We solve the optimization problem in both algorithms by adding noise to $\bmy_{i,t}^\star$, that is, $\bmy_{i,t} = \bmy_{i,t}^\star + \rho_t \ones$, \dy{where $\rho_t = c_\rho \cdot \frac{1}{t^{3/2}}$ with $c_\rho \geq 0$ and $\ones$ is the vector with all its entries equal to one.}

\begin{figure}[htb]
\begin{center}
\rotatebox{360}{\scalebox{0.47}[0.45]{\includegraphics{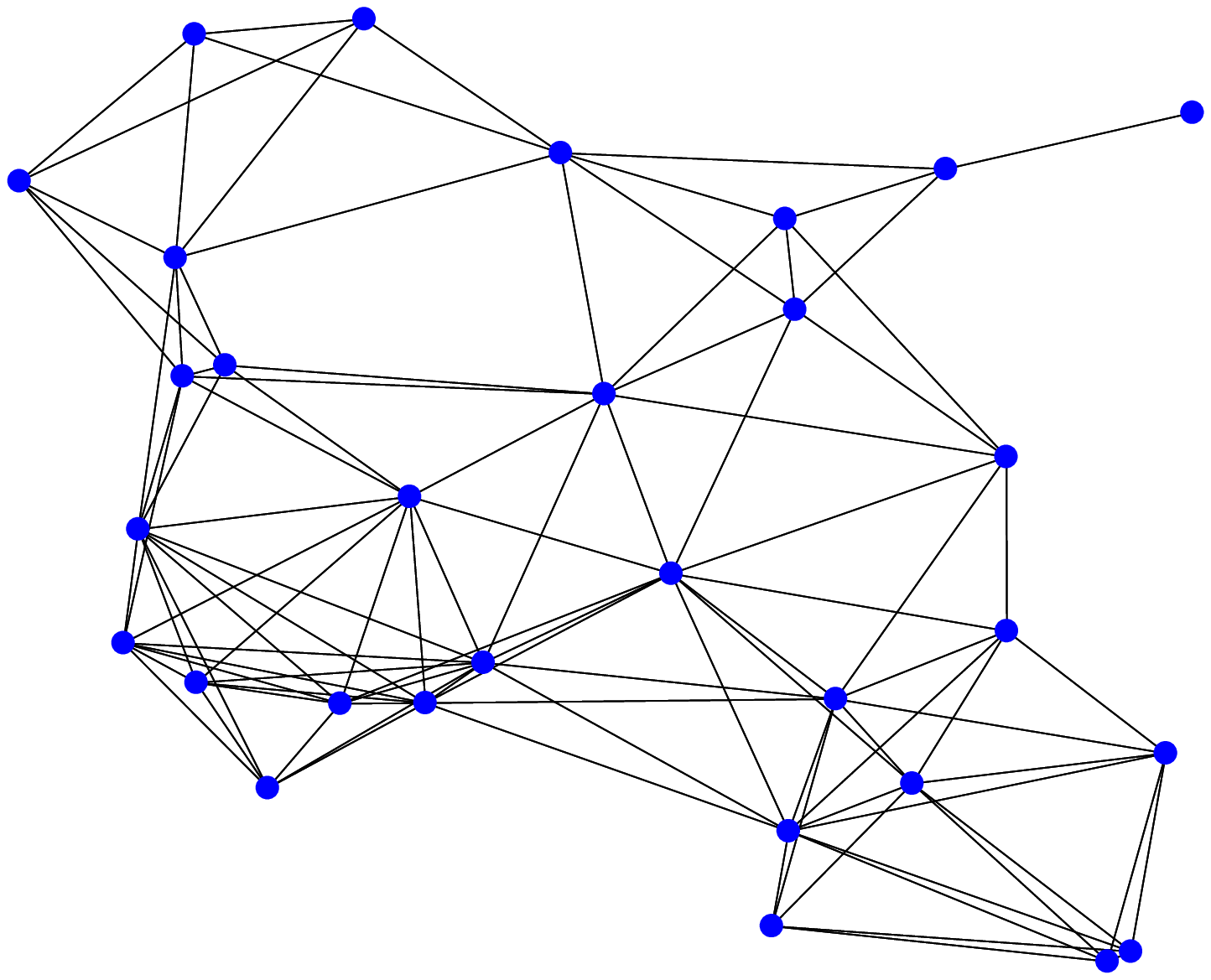}}}
\\[0pt]
{\normalsize {Fig. 1. A network of 30 nodes. } }
\end{center}
\end{figure}

We first consider solving problem (\ref{ridge-reg}) under full information feedback using $\mathsf{ODCMD}$. Let $c_\rho = 10$. We show the convergence performance of the algorithm by providing the plots of the maximum and minimum average regrets, defined respectively as \dy{$\max_{i\in [m]} \overline{\reg}_i(T)$ and $\min_{i\in [m]} \overline{\reg}_i(T) $}, versus the total number of rounds $T$. As a comparison, we also provide the convergence plots of a subgradient based algorithm; the details of the algorithm are as follows:
\begin{eqnarray*}
[ \widetilde{\bm{y}}_{i,t}  ]_s
&=&
\left\{
\begin{array}{l}
[ \bm{x}_{i,t} ]_s - \eta \left( [ \nabla_{i,t} ]_s + \lambda_2 \cdot {\rm sign} \left( [ \bm{x}_{i,t} ]_s \right)    \right),  \\
\hfill   {\rm if}\ [ \bm{x}_{i,t} ]_s \neq 0 \\
\null [ \bm{x}_{i,t} ]_s - \eta \left( [ \nabla_{i,t} ]_s + \lambda_2 \cdot c    \right), \hfill {\rm otherwise}
\end{array}
\right. \\
\bm{y}_{i,t}^{\star} &=&  \frac{\overline{R}}{\max\left(\| \widetilde{\bm{y}}_{i,t} \|_2 , \overline{R}  \right)} \widetilde{\bm{y}}_{i,t}
\end{eqnarray*}
where $c$ is an element of $[-1,1]$. \dy{The simulation results are shown in Fig. 2, which shows that algorithm $\mathsf{ODCMD}$ achieves better optimality than the subgradient based algorithm for each $T$.}

\begin{figure}[htb]
\begin{center}
\rotatebox{360}{\scalebox{0.47}[0.45]{\includegraphics{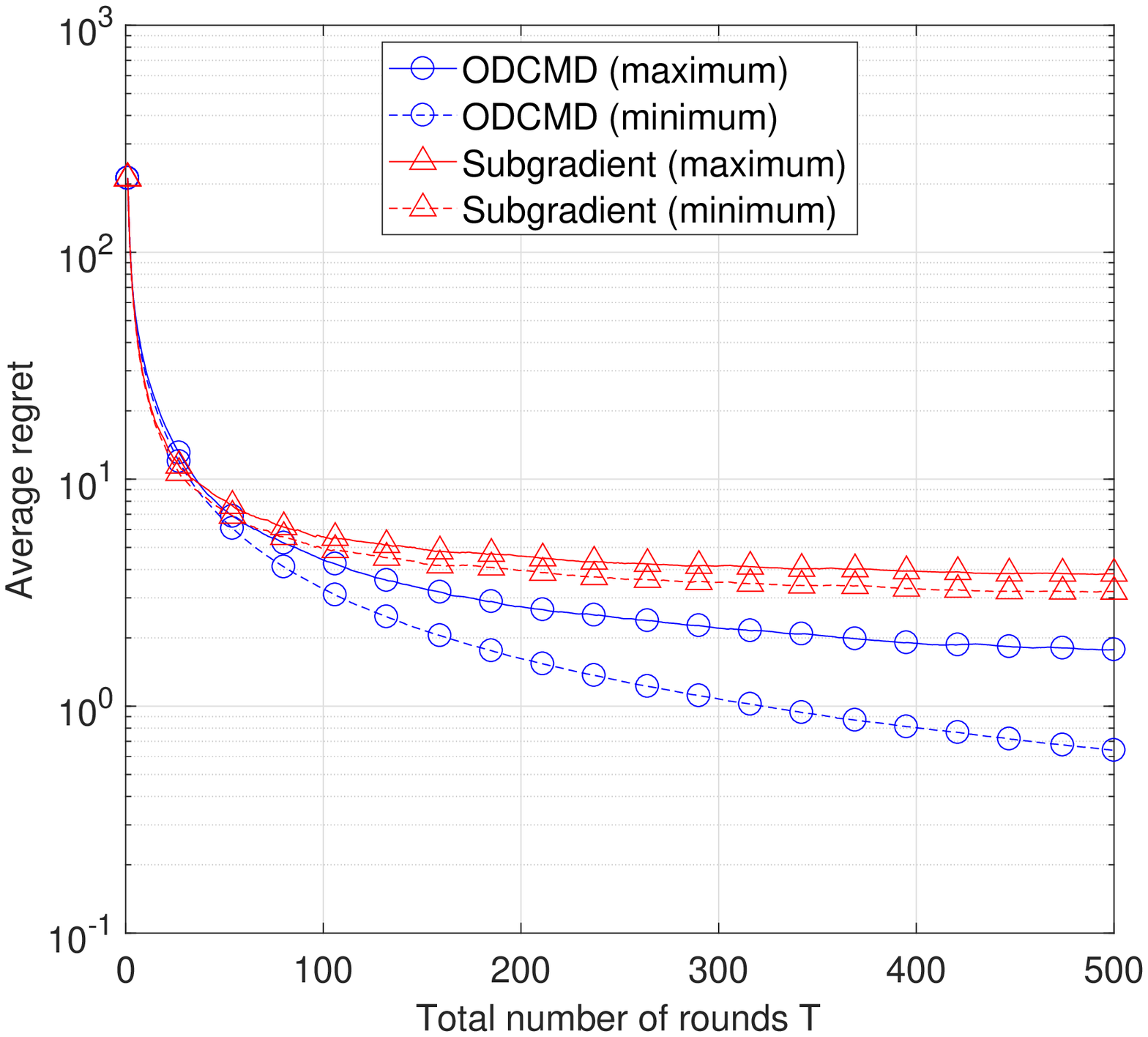}}}
\\[0pt]
{\normalsize {Fig. 2. The maximum and minimum average regrets versus the total number of rounds $T$ of $\mathsf{ODCMD}$ and subgradient based algorithm.  } }
\end{center}
\end{figure}

\dy{
We now study the effects of the optimization error sequence $\{\rho_t\}_{t=1}^{T}$ on the convergence of algorithm $\mathsf{ODCMD}$. Specifically, we first choose four different values of $c_\rho$ in $\rho_t$ in our simulations, \ie, $c_\rho = 0$ (error-free), $c_\rho = 10$, $c_\rho = 20$, and $c_\rho = 30$. As a comparison, we also consider two types of fixed optimization error sequences in simulations, namely, $\rho_t = 0.5$ and $\rho_t = \frac{10}{T^{3/2}}$ for all $t\in[T]$. The simulation results are shown in Fig. 3, which provides plots of the maximum average regret versus the total number of rounds $T$. It can be observed from Fig. 3 that: i) When the optimization error decreases as $\rho_t = \frac{c_\rho}{t^{3/2}}$, algorithm $\mathsf{ODCMD}$ achieves better optimality with a smaller value of $c_\rho$, and it is most obvious in the error-free case (\ie, $c_\rho = 0$); and ii) When the optimization error is fixed, algorithm $\mathsf{ODCMD}$ converges for the case of $\rho_t = \frac{10}{T^{3/2}}$, but when $\rho_t = 0.5$, algorithm $\mathsf{ODCMD}$ does not converge. All those observations comply with the results established in Theorem \ref{corollary-main-cvx}.
}

\begin{figure}[htb]
\begin{center}
\rotatebox{360}{\scalebox{0.47}[0.45]{\includegraphics{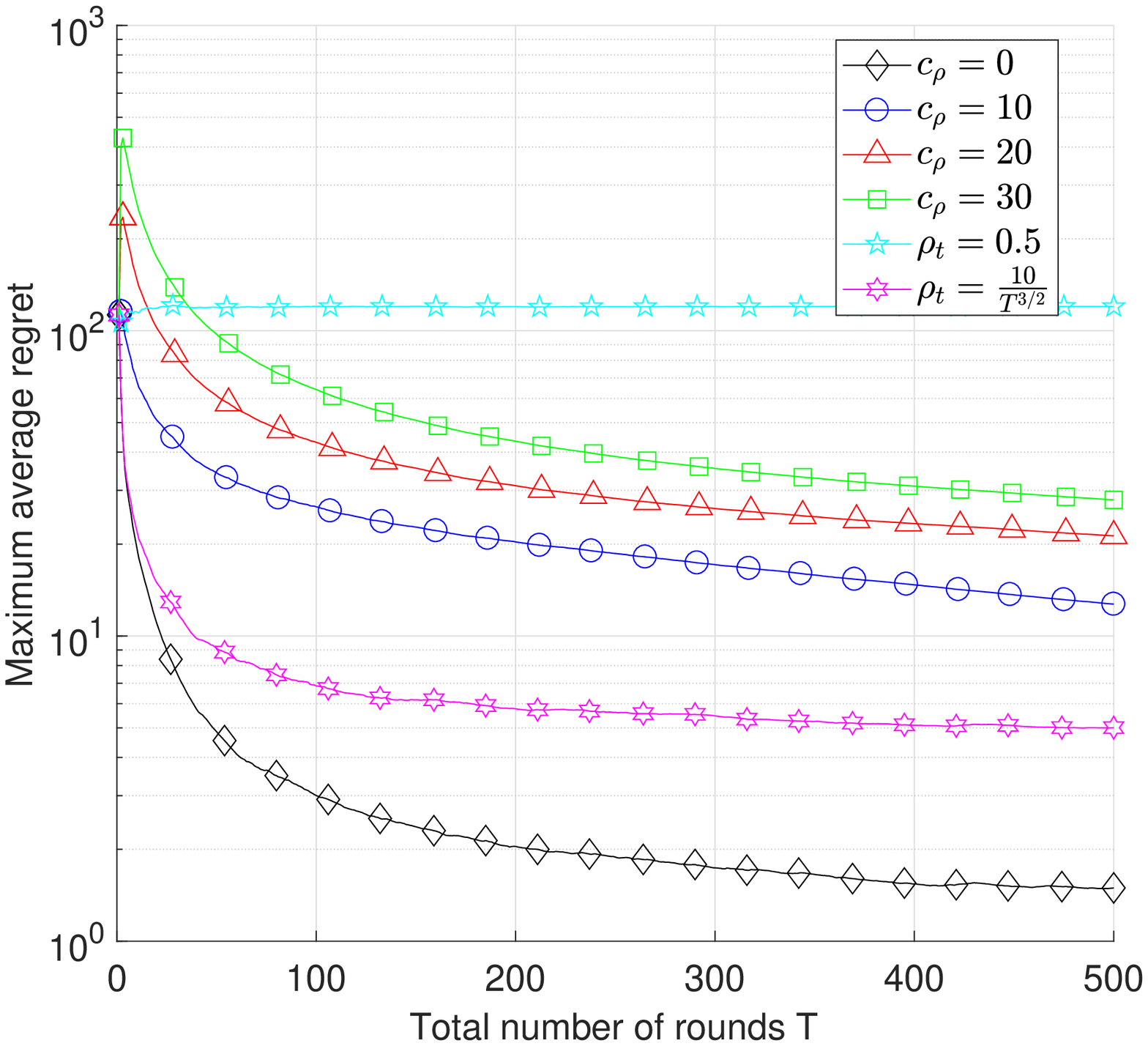}}}
\\[0pt]
{\normalsize {\dy{ Fig. 3. The maximum average regret versus the total number of rounds $T$ of $\mathsf{ODCMD}$ for different choices of $\rho_t$. } } }
\end{center}
\end{figure}

\dy{
We now make a comparison between algorithm $\mathsf{ODCMD}$ running with $\omega(\bmx) = \frac{1}{2} \| \bmx \|_2^2$ and algorithm $\mathsf{ODCMD}$ running with $\omega(\bmx) = \frac{1}{2} \| \bmx \|_p^2$. Let $\mathcal{K} = \reals^d$, $p = \frac{\ln (d)}{\ln(d)-1}$, $\lambda_1 =1$, $\lambda_2 = 0.1$, and $c_\rho = 10$. For the case of $p$-norm, step 3 in Algorithm 1 can be written explicitly as follows:
\BEASN
[ \nabla \omega(\bm{x}_{i,t}) ]_s &=& \frac{\mathrm{sign}\left( [ \bm{x}_{i,t} ]_s \right) | [ \bm{x}_{i,t} ]_s |^{p-1}}{\| \bm{x}_{i,t} \|_p^{p-2}}  \\
\widehat{\bm{y}}_{i,t} &=& \nabla \omega(\bm{x}_{i,t}) - \eta \nabla_{i,t} \\
\left[\overline{\bm{y}}_{i,t}\right]_s &=& \mathrm{sign}\left( [ \widehat{\bm{y}}_{i,t} ]_s \right)
\left[ | [ \widehat{\bm{y}}_{i,t} ]_s | - \eta \lambda \right]_+ \\
\left[\bm{y}_{i,t}^{\star}\right]_s &=& \frac{\mathrm{sign}\left( [ \overline{\bm{y}}_{i,t} ]_s \right) | [ \overline{\bm{y}}_{i,t} ]_s |^{q-1}}{\| \overline{\bm{y}}_{i,t} \|_q^{q-2}} ,  \qquad\qquad s\in[d].
\EEASN
The simulation results are presented in Fig. 4, which provides plots of the the maximum average regret versus $T$ of $\mathsf{ODCMD}$ using Euclidean distance and $\mathsf{ODCMD}$ using $p$-norm, for three different values of $d$, namely, $d=10$, $d=50$, and $d=100$. From Fig. 4 we observe that $p$-norm $\mathsf{ODCMD}$ achieves better optimality than $\mathsf{ODCMD}$ using Euclidean distance for each $T$, and this phenomenon is more obvious for a large value of $d$.
\begin{figure}[htb]
\begin{center}
\rotatebox{360}{\scalebox{0.414}[0.414]{\includegraphics{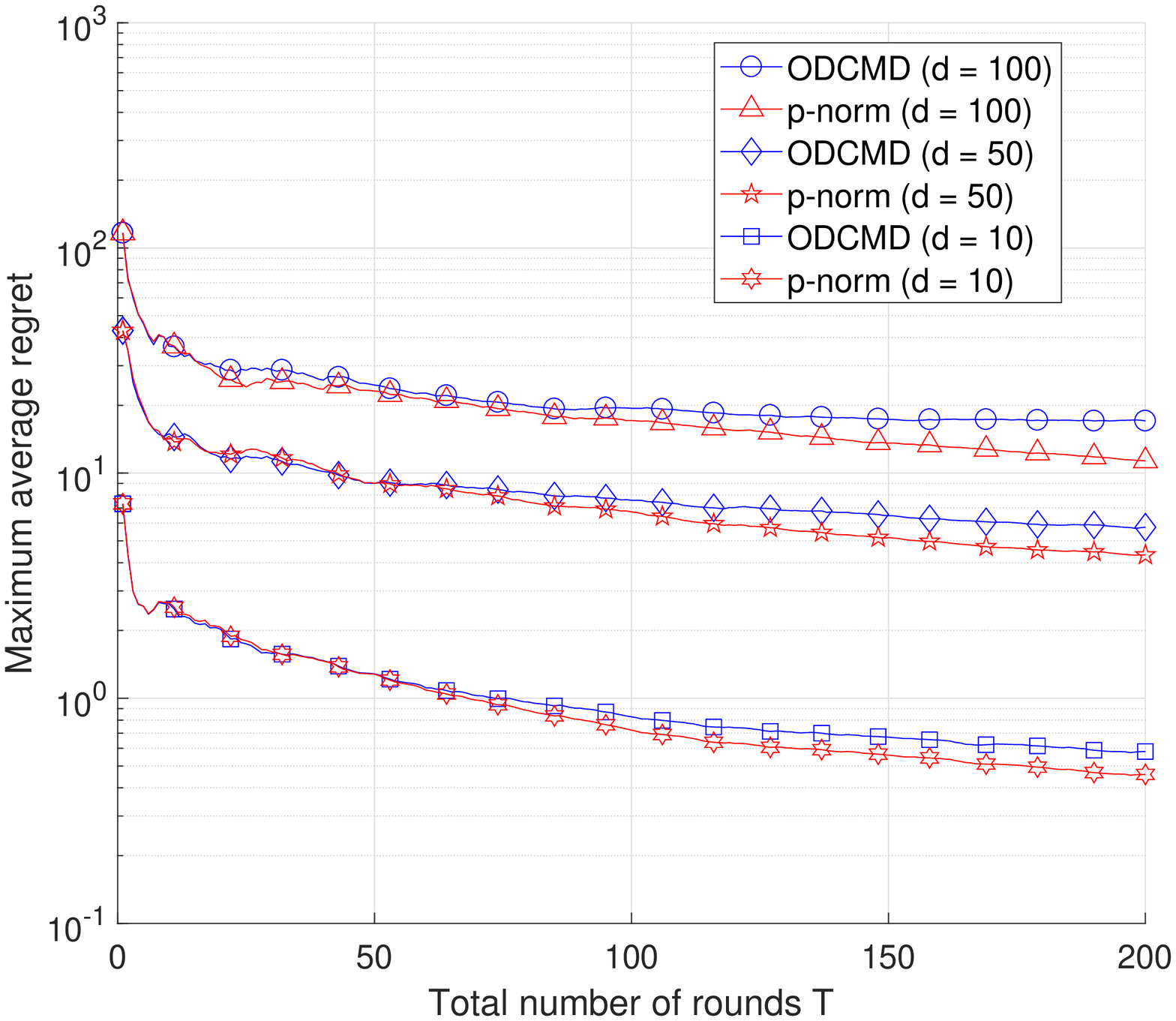}}}
\\[0pt]
{\normalsize {\dy{ Fig. 4. The maximum average regret versus the total number of rounds $T$ of $\mathsf{ODCMD}$ using Euclidean distance and $p$-norm $\mathsf{ODCMD}$ for three different choices of $d$. } } }
\end{center}
\end{figure}
}

We investigate the convergence performance of algorithm $\mathsf{BanODCMD}$. We use the same parameters as in the case of running algorithm $\mathsf{ODCMD}$ with Euclidean distance. Furthermore, let $\delta = \xi = \frac{1}{\sqrt{T}}$ and $c_\rho = 10$ in the simulations. First, we show the convergence of $\mathsf{BanODCMD}$ by providing the plots of the maximum and minimum average regrets of all the nodes versus the total number of rounds $T$. The simulation results in Fig. 5 show that both the the maximum and minimum average regrets of algorithm $\mathsf{BanODCMD}$ converges.

\begin{figure}[htb]
\begin{center}
\rotatebox{360}{\scalebox{0.47}[0.45]{\includegraphics{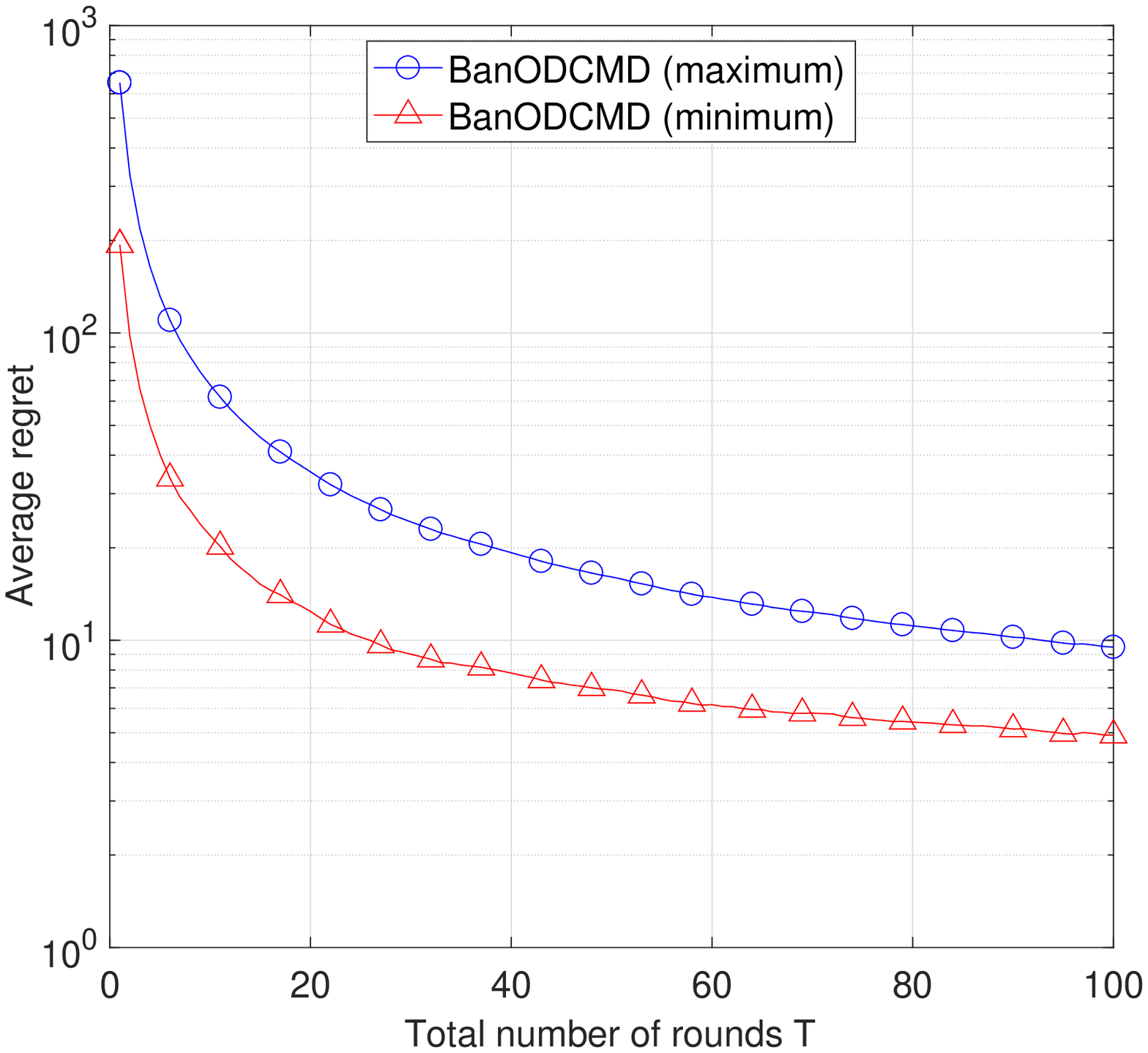}}}
\\[0pt]
{\normalsize {Fig. 5. The maximum and minimum average regrets versus the total number of rounds $T$ of $\mathsf{BanODCMD}$.  } }
\end{center}
\end{figure}

We next investigate the effect of the size of the network (\ie, number of nodes $m$) on the convergence of algorithm $\mathsf{BanODCMD}$. Specifically, we provide plots of the maximum average regrets versus the total number of rounds $T$, for three different number of nodes $m$ in a ring network, namely, $m=10$, $m=20$, and $m=30$, respectively. The simulation results are shown in Fig. 6, which reveals that algorithm $\mathsf{BanODCMD}$ achieves better optimality with a network of smaller size.

\begin{figure}[htb]
\begin{center}
\rotatebox{360}{\scalebox{0.47}[0.45]{\includegraphics{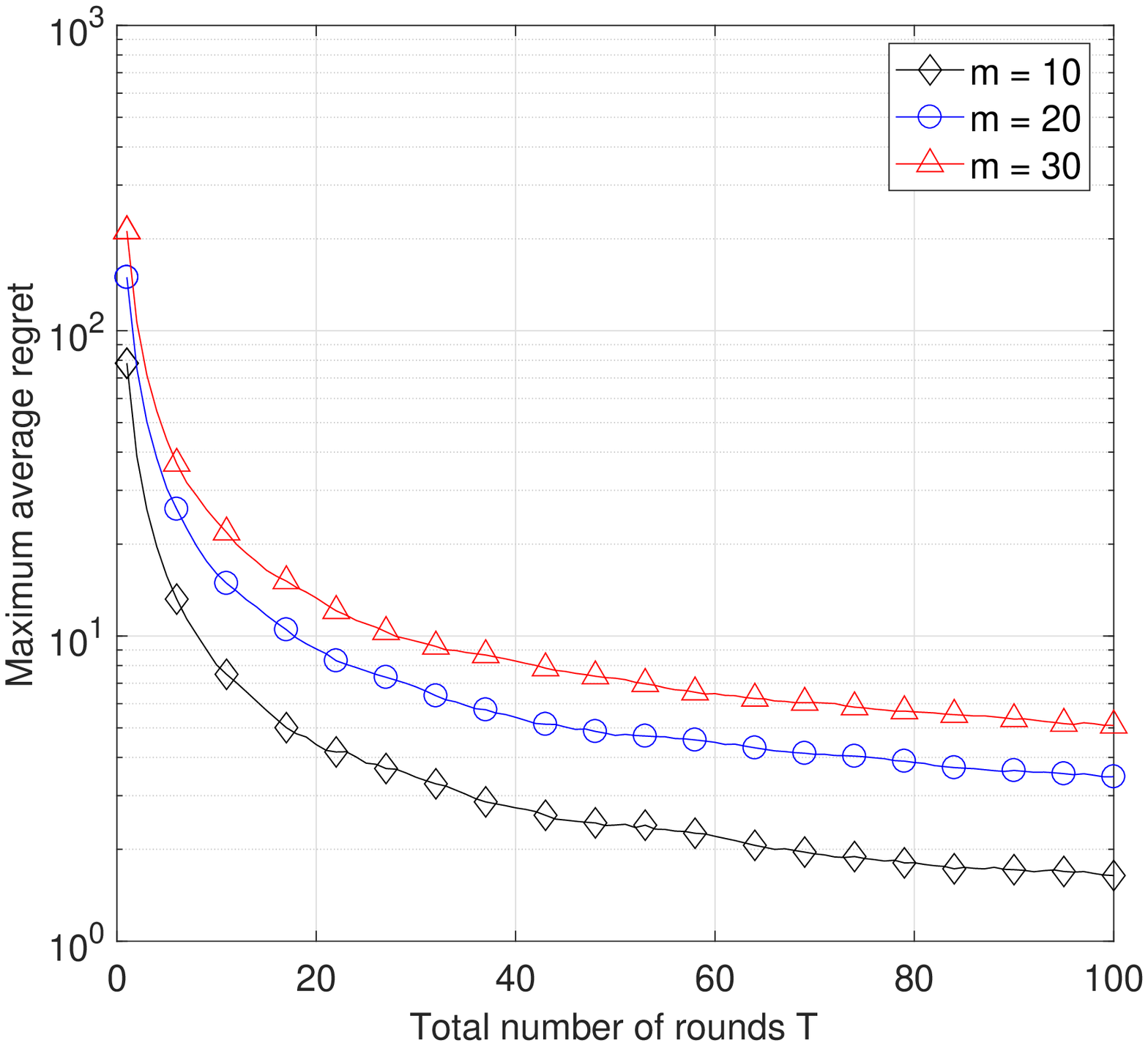}}}
\\[0pt]
{\normalsize {Fig. 6. The maximum average regret versus the total number of rounds $T$ of $\mathsf{BanODCMD}$ for three different number of nodes $m$ in a ring network.  } }
\end{center}
\end{figure}

We finally investigate the effects of the problem dimension $d$ on the convergence of algorithm {\sf BanODCMD}. Specifically, we provide plots of the maximum average regrets versus the total number of rounds $T$ for four different choices of the problem dimension $d$, \ie, $d=10$, $d=20$, $d=30$, and $d=40$. The simulation results are provided in Fig. 7, which clearly shows that algorithm $\mathsf{BanODCMD}$ achieves better optimality with small problem dimension $d$.
\pr{In addition, we have zoomed in on $T$ in the range $[60,100]$ and displayed the maximum average regret on a {\it linear} scale. It can be seen from Fig. 7 that the maximum average regret increases approximately linearly with increasing $d$. This is in compliance with the average regret scaling stated in Corollary 2, that is, $\mathcal{O} ( d/\sqrt{T} )$ ($\overline{p} = \overline{p}_{\ast} = 1$, due to the Euclidean norm).
}

\begin{figure}[htb]
\begin{center}
\rotatebox{360}{\scalebox{0.414}[0.414]{\includegraphics{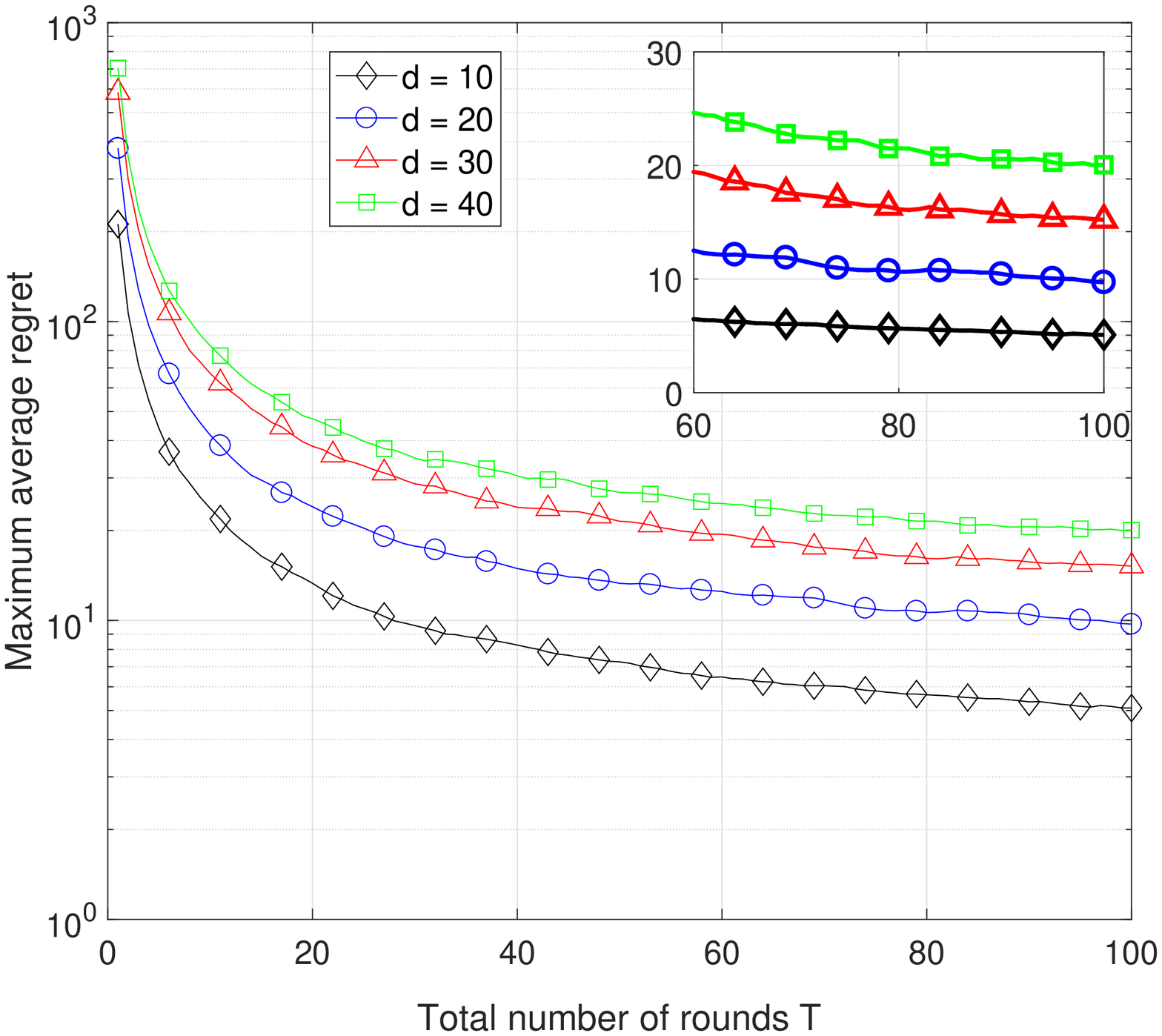}}}
\\[0pt]
{\normalsize {\pr{Fig. 7. The maximum average regret versus the total number of rounds $T$ of $\mathsf{BanODCMD}$ for four different choices of the problem dimension $d$.  } } }
\end{center}
\end{figure}


\section{Conclusion}
We have considered the problem of solving online distributed composite optimization over a network that consists of multiple interacting nodes. We have proposed two efficient online distributed composite mirror descent algorithms. The first algorithm has solved the problem under full-information feedback, and our second algorithm has solved the problem under bandit feedback where the information of the gradient is not available. For both algorithms, we have derived the \dy{average regularized regret is of order $\mathcal{O}(1/\sqrt{T})$}, which matches the previously known rates of centralized setting. We have also showed the effectiveness of our algorithms by implementing them over a distributed online regularized linear regression problem. This paper leaves several interesting questions. For example, it would be interesting to establish order-optimal upper and lower bounds for the proposed algorithms. For the case of bandit feedback, it would be of interest to obtain optimal dependence with the problem dimension $d$. Finally, it would be interesting to apply the algorithms to distributed optimization of different models.

\appendices
\section{Proof of Lemma \ref{theorem-basic-conv}}\label{appen:theorem:basic:conv}
Applying the first-order optimality condition, that is, $\left<  \grad f (\bmx^\star)  ,  \bmx -  \bmx^\star \right> \geq 0 $, for all $\bmx \in \calK$, where $\bmx^\star  =  \arg \min_{\bmx\in\calK} f (\bmx)$, to the optimization problem in step 3 in Algorithm \ref{alg-dcmd}, we can obtain that for all $\bmx\in\calK$,
\BEAS
\left<  \grad_{{}i,t} \!+\! \grad r(\bmy_{i,t}^\star) \!+\! \frac{1}{\eta} \!\! \left(  \grad\omega(\bmy_{i,t}^\star) \!-\!  \grad\omega(\bmx_{i,t}) \right) \!,\!  \bmx \!-\! \bmy_{i,t}^\star \right> \geq 0
\label{1st-order-induced}
\EEAS
since $\grad \breg_\omega (\bmx,\bmx_{i,t}) = \grad\omega(\bmx) -  \grad\omega(\bmx_{i,t})$.
By setting $\bmx = \bmx^\star$ in inequality (\ref{1st-order-induced}) it follows that
\BEASN
\left<  \grad_{i,t} \!+\! \grad r(\bmy_{i,t}^\star)  \!+\! \frac{1}{\eta} \left(  \grad\omega(\bmy_{i,t}^\star) - \grad\omega(\bmx_{i,t}) \right) , \bmy_{i,t}^\star - \bmx^\star  \right>  \leq 0
\EEASN
or equivalently,
\BEAS
&&\left<  \grad_{i,t} + \grad r(\bmy_{i,t}^\star)  ,  \bmy_{i,t}^\star -   \bmx^\star  \right>   \nn\\
&\leq& \frac{1}{\eta} \left<  \grad\omega(\bmy_{i,t}^\star) - \grad\omega(\bmx_{i,t}) , \bmx^\star -  \bmy_{i,t}^\star \right>.
\label{lemma-basic-conv-1}
\EEAS
By adding and subtracting $\bmy_{i,t}$ and using the Cauchy-Schwarz inequality, it follows that
\BEAS
&&\left<  \grad\omega(\bmy_{i,t}^\star) - \grad\omega(\bmx_{i,t}) , \bmx^\star -  \bmy_{i,t}^\star \right> \nn\\
&=& \left<  \grad\omega(\bmy_{i,t}^\star) - \grad\omega(\bmx_{i,t}) , \bmx^\star -  \bmy_{i,t} \right>  \nn\\
&&+ \left<  \grad\omega(\bmy_{i,t}^\star) - \grad\omega(\bmx_{i,t}) , \bmy_{i,t} -  \bmy_{i,t}^\star \right> \nn\\
&=& \left<  \grad\omega(\bmy_{i,t}) - \grad\omega(\bmx_{i,t}) , \bmx^\star -  \bmy_{i,t} \right>  \nn\\
&&+ \left<  \grad\omega(\bmy_{i,t}^\star) - \grad\omega(\bmy_{i,t}) , \bmx^\star -  \bmy_{i,t} \right>  \nn\\
&&+ \left<  \grad\omega(\bmy_{i,t}^\star) - \grad\omega(\bmx_{i,t}) , \bmy_{i,t} -  \bmy_{i,t}^\star \right> \nn\\
&\leq& \left<  \grad\omega(\bmy_{i,t}) - \grad\omega(\bmx_{i,t}) , \bmx^\star -  \bmy_{i,t} \right>  \nn\\
&&+ \| \bmx^\star -  \bmy_{i,t}  \| \cdot \| \grad\omega(\bmy_{i,t}^\star) - \grad\omega(\bmy_{i,t}) \|_\ast \nn\\
&&+ \| \bmy_{i,t} -  \bmy_{i,t}^\star  \| \cdot \| \grad\omega(\bmy_{i,t}^\star) - \grad\omega(\bmx_{i,t}) \|_\ast
\label{lemma-basic-conv-2}
\EEAS
which further yields the following bound, because $\omega$ has $G_\omega$-Lipschitz gradients (cf. Assumption \ref{assump-fcn-phi}),
\BEAS
&&\left<  \grad\omega(\bmy_{i,t}^\star) - \grad\omega(\bmx_{i,t}) , \bmx^\star -  \bmy_{i,t}^\star \right> \nn\\
&\leq& \left<  \grad\omega(\bmy_{i,t}) - \grad\omega(\bmx_{i,t}) , \bmx^\star -  \bmy_{i,t} \right>  \nn\\
&&+ G_\omega \left( \| \bmx^\star -  \bmy_{i,t}  \|  +  \| \bmy_{i,t}^\star -  \bmx_{i,t}  \| \right) \| \bmy_{i,t} -  \bmy_{i,t}^\star  \|  \nn\\
&\leq& \left<  \grad\omega(\bmy_{i,t}) - \grad\omega(\bmx_{i,t}) , \bmx^\star -  \bmy_{i,t} \right>  \nn\\
&&+ 2 G_\omega D_\calK \| \bmy_{i,t} -  \bmy_{i,t}^\star \|
\label{lemma-basic-conv-3}
\EEAS
where the last inequality follows from the fact that $\calK$ has finite diameter $D_\calK$. Combining the inequalities in (\ref{lemma-basic-conv-1}) and (\ref{lemma-basic-conv-3}), we get
\BEAS
&&\left<  \grad_{i,t} + \grad r(\bmy_{i,t}^\star)  ,  \bmy_{i,t}^\star -   \bmx^\star  \right>   +  \left<  \grad_{i,t} , \bmx_{i,t} -  \bmy_{i,t} \right>  \nn\\
&\leq& \frac{1}{\eta} \left<  \grad\omega(\bmy_{i,t}) - \grad\omega(\bmx_{i,t}) , \bmx^\star -  \bmy_{i,t} \right>  \nn\\
&&+ \frac{2}{\eta} G_\omega D_\calK  \| \bmy_{i,t} -  \bmy_{i,t}^\star \| + \left<  \grad_{i,t} , \bmx_{i,t} -  \bmy_{i,t}  \right>  \nn\\
&\leq&  \frac{1}{\eta}  \left( \breg_\omega (\bmx^\star , \bmx_{i,t})  -  \breg_\omega (\bmy_{i,t}, \bmx_{i,t})  - \breg_\omega (\bmx^\star , \bmy_{i,t}) \right) \nn\\
&&+ \frac{2}{\eta} G_\omega D_\calK   \| \bmy_{i,t} -  \bmy_{i,t}^\star \|
+ \frac{\sigma_\omega}{2\eta} \| \bmx_{i,t} -  \bmy_{i,t} \|^2  \nn\\
&&+ \frac{\eta}{2\sigma_\omega} \|  \grad_{i,t} \|^2_\ast
\label{lemma-basic-conv-4}
\EEAS
where the last inequality follows from the Pythagorean theorem for the Bregman divergence, \ie, $\left<  \grad \omega(\bmx) - \grad \omega(\bmz)  ,  \bmy  -  \bmz  \right>  =  \breg_\omega (\bmy,\bmz) +  \breg_\omega (\bmz,\bmx) -  \breg_\omega (\bmy,\bmx)  $ and the Fenchel-Young inequality, \ie, $\left< \bmx, \bmy\right> \leq a \| \bmx \|^2 + \frac{1}{a} \| \bmy \|^2_\ast$. Using the convexity of functions $\ell_{i,t}(\bmx)$ and $r(\bmx)$, the left-hand side of (\ref{lemma-basic-conv-4}) is lower bounded by
\BEAS
&&\left<  \grad_{i,t} + \grad r(\bmy_{i,t}^\star)  ,  \bmy_{i,t}^\star -   \bmx^\star  \right> +  \left<  \grad_{i,t} , \bmx_{i,t} -  \bmy_{i,t} \right>   \nn\\
&=& \left<  \grad_{i,t} + \grad r(\bmy_{i,t}^\star)  ,  \bmy_{i,t}^\star -   \bmx^\star  \right> +  \left<  \grad_{i,t} , \bmx_{i,t} -  \bmy_{i,t}^\star  \right>   \nn\\
&&  + \left<  \grad_{i,t} , \bmy_{i,t} -  \bmy_{i,t}^\star  \right> \nn\\
&\geq& \ell_{i,t}(\bmx_{i,t}) - \ell_{i,t}(\bmx^\star) + r(\bmy_{i,t}^\star) - r(\bmx^\star) \nn\\
&&- \| \bmy_{i,t} -  \bmy_{i,t}^\star \|  \cdot \| \grad_{i,t}  \|_\ast.
\label{lemma-basic-conv-5}
\EEAS
Combining the inequalities in (\ref{lemma-basic-conv-4}) and (\ref{lemma-basic-conv-5}) and using Assumption \ref{assump-fcn-obj}, we have
\BEAS
&&\ell_{i,t}(\bmx_{i,t})  + r(\bmy_{i,t}^\star)-\left(  \ell_{i,t}(\bmx^\star) + r(\bmx^\star)  \right) \nn\\
&\leq&  \frac{1}{\eta}  \left( \breg_\omega (\bmx^\star , \bmx_{i,t})  -  \breg_\omega (\bmy_{i,t}, \bmx_{i,t})  - \breg_\omega (\bmx^\star , \bmy_{i,t}) \right) \nn\\
&&+ \left( G_\ell +  \frac{2 }{\eta} G_\omega D_\calK \right)  \| \bmy_{i,t} -  \bmy_{i,t}^\star \|   +  \frac{\eta}{2\sigma_\omega} G_\ell^2 \nn\\
&&+ \frac{\sigma_\omega}{2\eta} \| \bmx_{i,t} -  \bmy_{i,t} \|^2  \nn\\
&\leq&  \frac{1}{\eta}  \left( \breg_\omega (\bmx^\star , \bmx_{i,t})  - \breg_\omega (\bmx^\star , \bmy_{i,t})  \right) \nn\\
&&+ \left( G_\ell +   \frac{2 }{\eta} G_\omega D_\calK \right)  \| \bmy_{i,t} -  \bmy_{i,t}^\star \|   +  \frac{G_\ell^2 }{2\sigma_\omega} \eta
\label{lemma-basic-conv-6}
\EEAS
due to the strong convexity of the Bregman divergence, \ie,
$$\breg_\omega (\bmy_{i,t} , \bmx_{i,t} )  \geq \frac{\sigma_\omega}{2}  \| \bmx_{i,t} -  \bmy_{i,t} \|^2.$$
Summing the inequalities in (\ref{lemma-basic-conv-6}) over $i = 1,\ldots,m$, gives
\BEAS
\hspace{-1.5em} &&\sum_{i=1}^{m} \left[ \ell_{i,t}(\bmx_{i,t})  + r(\bmy_{i,t}^\star)-\left(  \ell_{i,t}(\bmx^\star) + r(\bmx^\star)  \right) \right]  \nn\\
\hspace{-1.5em} &\leq&  \frac{1}{\eta} \sum_{i=1}^{m} \left( \breg_\omega (\bmx^\star , \bmx_{i,t})  - \breg_\omega (\bmx^\star , \bmy_{i,t})  \right)  \nn\\
\hspace{-1.5em} && + \left( G_\ell + \frac{2}{\eta} G_\omega D_\calK  \right)  \sum_{i=1}^{m} \| \bmy_{i,t} -  \bmy_{i,t}^\star \|   +  \frac{m G_\ell^2 }{2\sigma_\omega} \eta.
\label{lemma-basic-conv-7}
\EEAS
We now relate the left-hand side of (\ref{lemma-basic-conv-7}) and the average regularized regret (\ref{regularized-regret}). Using the Lipschitz continuity of function $\ell_{i,t}(\bmx)$ (cf. Assumption \ref{assump-fcn-obj}), we have
\BEAS
\ell_{i,t} (\bmx_{i,t}) &=& \ell_{i,t} (\bmx_{j,t}) + \ell_{i,t} (\bmx_{i,t}) - \ell_{i,t} (\bmx_{j,t}) \nn\\
&\geq& \ell_{i,t} (\bmx_{j,t}) - G_\ell \| \bmx_{i,t} -  \bmx_{j,t} \|
\label{lemma-basic-conv-7a}
\EEAS
and similarly, it follows from the Lipschitz continuity of function $r(\bmx)$ that
\BEAS
r(\bmy_{i,t}^\star) &=& r (\bmx_{j,t}) + r (\bmx_{i,t}) - r (\bmx_{j,t}) \nn\\
&&+ r(\bmy_{i,t}^\star) - r (\bmx_{i,t}) \nn\\
&\geq& r (\bmx_{j,t}) -  G_r \| \bmx_{i,t} -  \bmx_{j,t} \| \nn\\
&&- G_r \| \bmy_{i,t}^\star -  \bmx_{i,t} \|.
\label{lemma-basic-conv-7b}
\EEAS
Combining the preceding inequalities with (\ref{lemma-basic-conv-7}), using the definition of the average regularized regret in (\ref{regularized-regret}), and then summing the inequalities over all $t=1\cdots,T$, we arrive at
\BEAS
&&\sum_{t=1}^{T} \sum_{i=1}^{m} \left( \ell_{i,t}(\bmx_{j,t})  + r(\bmx_{j,t}) \right) \nn\\
&&- \sum_{t=1}^{T}  \sum_{i=1}^{m} \left(  \ell_{i,t}(\bmx^\star) + r(\bmx^\star)  \right)  \nn\\
&\leq&  \frac{1}{\eta} \sum_{t=1}^{T} \sum_{i=1}^{m} \left( \breg_\omega (\bmx^\star , \bmx_{i,t})  - \breg_\omega (\bmx^\star , \bmy_{i,t})  \right)  \nn\\
&&+ G_r  \sum_{t=1}^{T} \sum_{i=1}^{m}  \| \bmy_{i,t}^\star -  \bmx_{i,t} \| +  \frac{m G_\ell^2 }{2\sigma_\omega} \eta  T  \nn\\
&&+ \left( G_\ell + \frac{2}{\eta} G_\omega D_\calK  \right) \sum_{t=1}^{T} \sum_{i=1}^{m} \| \bmy_{i,t} -  \bmy_{i,t}^\star \|  \nn\\
&&+ \left( G_\ell + G_r \right)  \sum_{t=1}^{T} \sum_{i=1}^{m} \| \bmx_{i,t} -  \bmx_{j,t} \| .
\label{theorem-main-cvx-1}
\EEAS
We now bound the term $\sum_{i=1}^{m} \breg_\omega (\bmx^\star , \bmx_{i,t})$ as the following, by using the doubly stochasticity of the weight matrix $\wm(t-1)$ and Assumption \ref{assump-breg}, that is, for all $t\geq 2$,
\BEAS
\sum_{i=1}^{m} \breg_\omega (\bmx^\star , \bmx_{i,t})
&=& \sum_{i=1}^{m} \breg_\omega \left( \bmx^\star, \sum_{j=1}^{m} [\wm(t-1)]_{ij} \bmy_{j,t-1} \right) \nn\\
&\leq&    \sum_{i=1}^{m} \sum_{j=1}^{m} [\wm(t-1)]_{ij} \breg_\omega (\bmx^\star , \bmy_{j,t-1})  \nn\\
&=&   \sum_{j=1}^{m}  \breg_\omega (\bmx^\star , \bmy_{j,t-1}).
\label{theorem-main-cvx-2}
\EEAS
Hence, the first term on the right-hand side of (\ref{theorem-main-cvx-1}) leads to a telescopic sum, that is,
\BEAS
&&\sum_{t=1}^{T} \sum_{i=1}^{m} \left( \breg_\omega (\bmx^\star , \bmx_{i,t})  - \breg_\omega (\bmx^\star , \bmy_{i,t})  \right)    \nn\\
&\leq& \sum_{i=1}^{m} \left( \breg_\omega (\bmx^\star , \bmx_{i,1})  - \breg_\omega (\bmx^\star , \bmy_{i,1})  \right) \nn\\
&&+ \sum_{t=2}^{T} \sum_{i=1}^{m} \left( \breg_\omega (\bmx^\star , \bmy_{i,t-1})  - \breg_\omega (\bmx^\star , \bmy_{i,t})  \right)  \nn\\
&=& \sum_{i=1}^{m} \left( \breg_\omega (\bmx^\star , \bmx_{i,1}) - \breg_\omega (\bmx^\star , \bmy_{i,T}) \right)
\label{theorem-main-cvx-3}
\EEAS
which, combined with (\ref{theorem-main-cvx-1}), gives
\dy{
\BEASN
\overline{\reg}_{j}(T)
&\leq&  \frac{1}{\eta T} \sum_{i=1}^{m} \breg_\omega (\bmx^\star , \bmx_{i,1}) +  \frac{m G_\ell^2 }{2\sigma_\omega} \eta \nn\\
&&+ G_r \frac{1}{T} \sum_{t=1}^{T} \sum_{i=1}^{m}  \| \bmy_{i,t}^\star -  \bmx_{i,t} \|  \nn\\
&&+ \left( G_\ell +  \frac{2 }{\eta} G_\omega D_\calK \right) \frac{1}{T} \sum_{t=1}^{T} \sum_{i=1}^{m} \| \bmy_{i,t} -  \bmy_{i,t}^\star \|  \nn\\
&&+ \left( G_\ell + G_r \right)  \frac{1}{T} \sum_{t=1}^{T} \sum_{i=1}^{m} \| \bmx_{i,t} -  \bmx_{j,t} \|
\EEASN
}
with recalling the definition of the average regularized regret and dropping the negative term $-\sum_{i=1}^{m} \breg_\omega (\bmx^\star , \bmy_{i,T}) $.
$\hfill\blacksquare$

\section{Proof of Lemma \ref{lemma-disagreement}}\label{appen:lemma:disagreement}
(a) To facilitate the analysis, we introduce three auxiliary variables for all $i\in[m]$ and $t\in[T]$ as follows:
\BEAS
\overline{\bmx}_t &=& \frac{1}{m} \sum_{i=1}^{m} \bmx_{i,t} \\
\bme_{i,t} &=& \bmy_{i,t} - \bmy_{i,t}^\star
\label{theorem-main-convex-6-aux-1}      \\
\bm{\epsilon}_{i,t} &=& \bmy_{i,t}^\star - \bmx_{i,t}.
\label{lemma-disagreement-1}
\EEAS
We first utilize the fact that function $\left< \grad_{i,t}, \bmx \right>  + r(\bmx) + \frac{1}{\eta} \breg_\omega (\bmx,\bmx_{i,t} ) $ is $\frac{\sigma_\omega}{\eta}$-strongly convex to obtain
\BEAS
&& \left< \grad_{i,t}, \bmy_{i,t} \right>  + r(\bmy_{i,t}) + \frac{1}{\eta} \breg_\omega (\bmy_{i,t},\bmx_{i,t} ) \nn\\
&\geq& \left< \grad_{i,t}, \bmy_{i,t}^\star \right>  + r(\bmy_{i,t}^\star ) + \frac{1}{\eta} \breg_\omega (\bmy_{i,t}^\star ,\bmx_{i,t} )  \nn\\
&&+ \frac{\sigma_\omega}{2 \eta}  \| \bmy_{i,t} -  \bmy_{i,t}^\star \|^2
\label{lemma-disagreement-7}
\EEAS
which, combined with step 3 in Algorithm \ref{alg-dcmd}, yields
\BEAS
\| \bme_{i,t} \| = \| \bmy_{i,t} -  \bmy_{i,t}^\star \|  &\leq& \sqrt{  \frac{2}{\sigma_\omega } \eta \rho_t } .
\label{lemma-disagreement-8}
\EEAS

(b) We now turn our attention to the term $\| \bmy_{i,t}^\star -  \bmx_{i,t} \|$. By setting $\bmx = \bmx_{i,t}$ in the first-order optimality condition in (\ref{1st-order-induced}) it follows that
\BEASN
&&\left<  \grad_{{}i,t} + \grad r(\bmy_{i,t}^\star) ,  \bmx_{i,t} - \bmy_{i,t}^\star \right>  \nn\\
&\geq& \left<  \frac{1}{\eta} \left(  \grad\omega(\bmy_{i,t}^\star) -  \grad\omega(\bmx_{i,t}) \right) , \bmy_{i,t}^\star - \bmx_{i,t} \right>  \nn\\
&\geq& \frac{\sigma_\omega}{\eta} \| \bmy_{i,t}^\star - \bmx_{i,t} \|^2
\label{lemma-disagreement-9}
\EEASN
because function $\omega$ is $\sigma_\omega$-strongly convex, and then applying the Cauchy-Schwarz inequality to the left-hand side, we have
\BEASN
\frac{\sigma_\omega}{\eta} \| \bmy_{i,t}^\star - \bmx_{i,t} \|^2 &\leq&  \| \bmy_{i,t}^\star - \bmx_{i,t} \| \cdot \| \grad_{{}i,t} + \grad r(\bmy_{i,t}^\star) \|_\ast
\label{lemma-disagreement-10}
\EEASN
which gives
\BEAS
\| \bm{\epsilon}_{i,t} \| = \| \bmy_{i,t}^\star - \bmx_{i,t} \| &\leq& \frac{1}{\sigma_\omega} \left(  G_\ell + G_r \right) \eta
\label{lemma-disagreement-11}
\EEAS
according to Assumption \ref{assump-fcn-obj}.

(c) We first derive the general iteration relation of the states $\bmx_{i,t+1}$,
\BEAS
\bmx_{i,t+1} &=& \sum_{j=1}^{m} [\wm(t)]_{ij} \bmy_{j,t} \nn\\
&=& \sum_{j=1}^{m} [\wm(t)]_{ij} \left(  \bmy_{j,t}^\star +  \bme_{j,t} \right)  \nn\\
&=& \sum_{j=1}^{m} [\wm(t)]_{ij} \left(  \bmx_{j,t} +  \bme_{j,t} + \bm{\epsilon}_{j,t} \right).
\label{lemma-disagreement-2}
\EEAS
Applying this inequality recursively, we find that for all $t\geq1$,
\BEAS
\bmx_{i,t+1} &=& \sum_{j=1}^{m} [\wm(t:1)]_{ij}  \bmx_{j,1} \nn\\
&&+ \sum_{\tau=1}^{t} \sum_{j=1}^{m}  [\wm(t:\tau)]_{ij} \left(  \bme_{j,\tau} + \bm{\epsilon}_{j,\tau} \right)
\label{lemma-disagreement-3}
\EEAS
where we write $\wm(t:\tau) = \wm(t) \wm(t-1) \cdots \wm(\tau+1) \wm(\tau)$ and $\wm(t:t) = \wm(t)$ for all $t\geq\tau\geq 1$. We now characterize the general iteration for the average states of the network, that is,
\BEASN
\overline{\bmx}_{t+1} &=& \frac{1}{m} \sum_{i=1}^{m} \bmx_{i,t+1} \nn\\
&=&  \frac{1}{m} \sum_{i=1}^{m}   \sum_{j=1}^{m} [\wm(t)]_{ij} \bmy_{j,t}  \nn\\
&=& \frac{1}{m} \sum_{i=1}^{m} \bmy_{i,t} \nn\\
&=& \frac{1}{m} \sum_{i=1}^{m} \left( \bmx_{i,t} +  \bme_{i,t} + \bm{\epsilon}_{i,t} \right)
\label{lemma-disagreement-4}
\EEASN
where the second-to-last equality follows from the double stochasticity of $\wm(t)$, and the last equality follows from equation (\ref{lemma-disagreement-2}). Applying this inequality recursively, we find that for all $t\geq1$,
\BEAS
\overline{\bmx}_{t+1} &=& \overline{\bmx}_{1} + \sum_{\tau=1}^{t} \frac{1}{m} \sum_{i=1}^{m}  \left( \bme_{i,\tau} + \bm{\epsilon}_{i,\tau} \right).
\label{lemma-disagreement-5}
\EEAS
Hence, for all $t\geq 1$ and any $i\in[m]$,
\BEAS
&&\| \bmx_{i,t+1} - \overline{\bmx}_{t+1} \|   \nn\\
&\leq& \sum_{j=1}^{m} \left| [\wm(t:1)]_{ij} - \frac{1}{m} \right| \cdot \| \bmx_{j,1} \|   \nn\\
&&+\sum_{\tau=1}^{t} \sum_{j=1}^{m} \left| [\wm(t:\tau)]_{ij} \!-\! \frac{1}{m} \right|  \left( \| \bme_{j,\tau} \|  \!+\! \| \bm{\epsilon}_{j,\tau} \| \right). \nn\\
\label{lemma-disagreement-6}
\EEAS
We now bound the norm of the differences of the estimates among nodes in the network, by combining the inequalities (\ref{lemma-disagreement-6}), (\ref{lemma-disagreement-8}) and (\ref{lemma-disagreement-11}) with Corollary 1 in \cite{nedic2008cdc} on the convergence properties of the matrix $\wm(t:\tau)$ for all $t\geq\tau\geq 1$, that is,
\BEASN
\left| [\wm(t:\tau)]_{ij} - \frac{1}{m} \right| \leq \vartheta \kappa^{t-\tau}.
\EEASN
Specifically, we have
\BEAS
\hspace{-1em} \sum_{i=1}^{m} \| \bmx_{i,t+1} - \overline{\bmx}_{t+1} \|
&\leq& \vartheta \left(\sum_{i=1}^{m} \| \bmx_{i,1} \| \right) \kappa^{t-1} \nn\\
&&\hspace{-1em}+ m \vartheta  \frac{ G_\ell + G_r }{\sigma_\omega} \eta \sum_{\tau=1}^{t} \kappa^{t-\tau}  \nn\\
&&\hspace{-1em}+ m \vartheta \sqrt{  \frac{2}{\sigma_\omega } } \sum_{\tau=1}^{t} \kappa^{t-\tau} \sqrt{\eta \rho_t }.
\label{lemma-disagreement-12}
\EEAS
On the other hand, one has
\BEASN
&&\sum_{t=1}^{T} \sum_{i=1}^{m} \| \bmx_{i,t} - \overline{\bmx}_{t} \| \nn\\
&=& \sum_{i=1}^{m} \| \bmx_{i,1} - \overline{\bmx}_{1} \| + \sum_{t=1}^{T-1} \sum_{i=1}^{m} \| \bmx_{i,t+1} - \overline{\bmx}_{t+1} \|  \nn\\
&\leq& \sum_{t=1}^{T} \sum_{i=1}^{m} \| \bmx_{i,t+1} - \overline{\bmx}_{t+1} \|
\EEASN
where the last inequality follows from the fact that $\bmx_{i,1}$ are the same for all $i$. Combining (\ref{lemma-disagreement-12}) with the preceding inequality, we get
\BEASN
\sum_{t=1}^{T} \sum_{i=1}^{m} \| \bmx_{i,t} - \overline{\bmx}_{t} \|
&\leq& \vartheta \left(\sum_{i=1}^{m} \| \bmx_{i,1} \| \right) \sum_{t=1}^{T} \kappa^{t-1} \nn\\
&&+  m \vartheta  \frac{ G_\ell + G_r }{\sigma_\omega} \eta \sum_{t=1}^{T} \sum_{\tau=1}^{t} \kappa^{t-\tau}  \nn\\
&&+ m \vartheta  \sqrt{  \frac{2}{\sigma_\omega } } \sum_{t=1}^{T} \sum_{\tau=1}^{t} \kappa^{t-\tau} \sqrt{\eta \rho_t }.
\EEASN
Therefore, the desired estimate follows by applying the following inequalities, that is,
\BEASN
\sum_{t=1}^{T}  \sum_{\tau=1}^{t} \kappa^{t-\tau}  &\leq& \sum_{t=1}^{T}  \left( \sum_{\tau=1}^{\infty} \kappa^{\tau} \right) \leq \frac{1}{1-\kappa} T
\EEASN
and
\BEASN
\sum_{t=1}^{T}  \sum_{\tau=1}^{t} \kappa^{t-\tau} \sqrt{\eta \rho_t } &\leq& \sum_{t=1}^{T} \!\!\left( \sum_{\tau=1}^{\infty} \kappa^{\tau} \right) \sqrt{\eta \rho_t } \\
&\leq& \frac{1}{1-\kappa} \sum_{t=1}^{T}\sqrt{\eta \rho_t}
\EEASN
to the preceding one, with the help of the triangle inequality $\| \bmx_{i,t} - \bmx_{j,t} \| \leq \| \bmx_{i,t} - \overline{\bmx}_{t} \| + \| \bmx_{j,t} - \overline{\bmx}_{t} \|$. The proof is complete.
$\hfill\blacksquare$



\begin{IEEEbiography}
{Deming Yuan} (M'17) received the B.Sc. degree in electrical engineering and automation and the Ph.D. degree in control science and engineering from the Nanjing University of Science and Technology, Nanjing, China, in 2007 and 2012, respectively.

He is currently a Professor with the School of Automation, Nanjing University of Science and Technology, Nanjing. He received the Australia Award Endeavour Research Fellowship to further his research on distributed optimization, and was hosted by the Australian National University, in 2018. He has been on the editorial board of \emph{Transactions of the Institute of Measurement and Control}. His current research interests include algorithms for distributed optimization and control.
\end{IEEEbiography}

\begin{IEEEbiography}
{Yiguang Hong} (F'17) received his B.S. and M.S. degrees from Peking University, China, and the Ph.D. degree from the Chinese Academy of Sciences (CAS), China. He is currently a Professor in Academy of Mathematics and Systems Science, CAS, and serves as the Director of Key Lab of Systems and Control, CAS and the Director of the Information Technology Division, National Center for Mathematics and Interdisciplinary Sciences, CAS.  His current research interests include nonlinear control, multi-agent systems, distributed optimization/game, machine learning, and social networks.

Prof. Hong serves as Editor-in-Chief of Control Theory and Technology and Deputy Editor-in-Chief of Acta Automatica Sinca. He also serves or served as Associate Editors for many journals, including the IEEE Transactions on Automatic Control, IEEE Transactions on Control of Network Systems, and IEEE Control Systems Magazine. He is a recipient of the Guang Zhaozhi Award at the Chinese Control Conference, Young Author Prize of the IFAC World Congress, Young Scientist Award of CAS, the Youth Award for Science and Technology of China, and the National Natural Science Prize of China.  He is a Fellow of IEEE.
\end{IEEEbiography}

\begin{IEEEbiography}
{Daniel W.C. Ho} (F'17) received the B.S., M.S., and Ph.D. degrees in mathematics from the University of Salford, Greater Manchester, U.K., in 1980, 1982, and 1986, respectively.

From 1985 to 1988, he was a Research Fellow with the Industrial Control Unit, University of Strathclyde, Glasgow, U.K. In 1989, He joined City University of Hong Kong, Hong Kong. He is currently a Chair Professor in applied mathematics with the Department of Mathematics, and the Associate Dean with the College of Science.  His current research interests include control and estimation theory, complex dynamical distributed networks, multi-agent networks, and stochastic systems. He has over 230 publications in scientific journals. Prof. Ho is a Fellow of the IEEE. He was honored to be the Chang Jiang Chair Professor awarded by the Ministry of Education, China, in 2012. He has been on the editorial board of a number of journals including \emph{IEEE Transactions on Neural Networks and Learning Systems}, \emph{IET Control Theory and its Applications}, \emph{Journal of the Franklin Institute} and \emph{Asian Journal of Control}. He is named as ISI Highly Cited Researchers (2014-2019) in Engineering by Clarivate Analytics.
\end{IEEEbiography}

\begin{IEEEbiography}
{Shengyuan Xu} received his B.Sc. degree from the Hangzhou Normal University, China in 1990, M.Sc. degree from the Qufu Normal University, China in 1996, and Ph.D. degree from the Nanjing University of Science and Technology, China 1999. From 1999 to 2000 he was a Research Associate in the Department of Mechanical Engineering at the University of Hong Kong, Hong Kong. From December 2000 to November 2001, and December 2001 to September 2002, he was a Postdoctoral Researcher in CESAME at the Universit\`{e} catholique de Louvain, Belgium, and the Department of Electrical and Computer Engineering at the University of Alberta, Canada, respectively. From September 2002 to September 2003, and September 2003 to September 2004, he was a William Mong Young Researcher and an Honorary Associate Professor, respectively, both in the Department of Mechanical Engineering at the University of Hong Kong, Hong Kong. Since November 2002, he has joined the School of Automation at the Nanjing University of Science and Technology as a professor.

Dr. Xu was a recipient of the National Excellent Doctoral Dissertation Award in the year 2002 from the Ministry of Education of China. He obtained a grant from the National Science Foundation for Distinguished Young Scholars of P. R. China In the year 2006. He was awarded a Cheung Kong Professorship in the year 2008 from the Ministry of Education of China.

Dr. Xu is a member of the Editorial Boards of the \emph{Transactions of the Institute of Measurement and Control}, and the \emph{International Journal of Control, Automation, and Systems}. His current research interests include robust filtering and control, singular systems, time-delay systems, neural networks, multidimensional systems and nonlinear systems.
\end{IEEEbiography}

\end{document}